\documentclass[11pt]{amsart}
\usepackage{amsmath,amssymb,amsthm,amsxtra,graphics,enumerate} 
\usepackage{stmaryrd}
\usepackage{times}
\usepackage[margin=1.3in]{geometry}
\usepackage[linktocpage=true, hidelinks, colorlinks=false, allcolors=blue]{hyperref}
\usepackage{comment}
\usepackage{scalerel}
\usepackage{mathrsfs}
\usepackage{accents}

\newtheorem{thm}{Theorem}[section]

\newtheorem{lemma}[thm]{Lemma}
\newtheorem{prop}[thm]{Proposition}

\theoremstyle{definition}
\newtheorem{defn}[thm]{Definition}

\newtheorem{probnumber}{Problem}
\newtheorem{thm2}{Theorem}

\newcommand\Bcal{\mathcal{B}}

\newcommand\Fcal{\mathcal{F}}

\newcommand\Lcal{\mathcal{L}}

\newcommand\Pcal{\mathcal{P}}

\newcommand\Ucal{\mathcal{U}}

\newcommand{\p}{\mathbb{P}}

\newcommand{\dom}{\mathrm{dom}}

\newcommand{\res}{\upharpoonright}
\newcommand\h{\text{ht}}

\begin{document}

\title[Countryman Lines and the Continuum Hypothesis]{Countryman Lines and the Continuum Hypothesis}

\author{John Krueger and Eduardo Martinez Mendoza}

\address{John Krueger, Department of Mathematics, 
	University of North Texas,
	1155 Union Circle \#311430,
	Denton, TX 76203, USA}
\email{john.krueger@unt.edu}

\address{Eduardo Martinez Mendoza, Department of Mathematics, 
	University of North Texas,
	1155 Union Circle \#311430,
	Denton, TX 76203, USA}
\email{eduardomartinezmendoza@my.unt.edu}



\date{July 25, 2026}

\begin{abstract}
	We explore the prospect of basis-like results for the class of Aronszajn lines 
	which are consistent with the Continuum Hypothesis (\textsf{CH}). 
	In particular, we prove that each of the following statements is consistent with 
	\textsf{CH}: Any two Countryman lines contain isomorphic or 
	anti-isomorphic uncountable suborders; 
	for any coherent Aronszajan tree 
	$T \subseteq {}^{\underaccent{\breve}{\omega}_1} \omega$, 
	the filter $\Ucal(T)$ is an ultrafilter. 
	The first result confirms a conjecture of Shelah in the context of \textsf{CH} which 
	was previously shown to follow from the Proper Forcing Axiom. 
	On the other hand, the weak diamond principle $2^\omega < 2^{\omega_1}$ implies that 
	there does not exist a two element basis for the Countryman lines.
\end{abstract}

\maketitle

\section{Introduction}

An \emph{Aronszajn line} is an uncountable linear order which contains no 
uncountable well-ordered or conversely well-ordered suborder nor an uncountable 
separable suborder. 
Aronszajn lines are closely connected to the idea of an \emph{Aronszajn tree}, which 
is a tree of height $\omega_1$ with countable levels and no uncountable chain. 
An uncountable linear order is an Aronszajn line iff it is the lexicographical 
ordering of some Aronszajn tree (\cite{baumgartner2, todortreeslinear}). 
A distinguished subclass of the Aronszajn lines are the \emph{Countryman lines}, which 
are uncountable linear orders whose square is a countable union of chains 
in the componentwise order. 
The existence of Countryman lines was proved by Shelah \cite{shelahcountryman}, 
solving a problem of Countryman \cite{countryman}. 
Countryman lines are particularly relevant to basis problems for uncountable linear 
orders.\footnote{A \emph{basis} for a class $\Lcal$ of linear orders is a subclass 
$\Bcal \subseteq \Lcal$ such that every member of $\Lcal$ contains a suborder 
which is isomorphic to a member of $\Bcal$.} 
Since a Countryman line cannot contain uncountable anti-isomorphic suborders, any basis 
for the Aronszajn lines must have size at least two. 
Shelah asked whether it is consistent that every Aronszajn 
line contains a Countryman line, a statement known as \emph{Shelah's conjecture}. 
Moore \cite{linear_basis} confirmed this conjecture by proving that the Proper Forcing Axiom 
implies that there exists a two element basis for the Aronszajn lines consisting of a 
Countryman line and its converse.\footnote{More broadly, when combined with work of 
Baumgartner \cite{baumgartner}, the Proper Forcing Axiom implies the existence of a five element 
basis for the uncountable linear orders.}

A notable feature of the two element basis theorem for Aronszajn lines, 
including both the work of Moore \cite{linear_basis} 
as well as the reduction of the problem to Shelah's conjecture (\cite{AS, lipschitz}), 
is that every forcing involved in the proof adds reals. 
As a result, the methods used to prove this theorem 
are antithetical to the construction of models of \textsf{CH}. 
This fact raises the question whether the violation of \textsf{CH} is essential, 
or if it is just a consequence of the specific techniques used. 
Relevant to this question is work of Abraham and Shelah \cite{AS}, who 
used the weak diamond principle $2^\omega < 2^{\omega_1}$, 
which follows from \textsf{CH}, to derive anti-basis results for Aronszajn trees. 
For instance, they showed that the weak diamond principle 
implies that for every Aronszajn tree $T$, 
there is an Aronzajn tree $U$ 
such that $T$ does not club embed into $U$. 
Other authors have used their method to derive related 
consequences of the weak diamond principle, including:
\begin{itemize}
	\item There does not 
	exist a two element basis for the Aronszajn lines (\cite[Cor.\ 4.11(a)]{baumgartner2}). 
	\item There exists an Aronszajn line which is the lexicographical ordering of a coherent 
	Aronszajn tree which is not \emph{minimal}, in the sense that it contains an uncountable 
	suborder into which it does not embed 
	(P. Larson, unpublished (see \cite[Sec.\ 8]{moorestructural})). 
	\item There does not exist a two element basis for the Countryman lines (\cite{emmdiss}). 
\end{itemize}
This situation suggests the question of whether, and to what extent, basis-like statements 
for the Aronszajn and Countryman lines are consistent with \textsf{CH}. 
To simplify language, we introduce some terminology.  
Uncountable linear orders $L$ and $M$ are said to be \emph{near} if there exist 
uncountable suborders $L_0 \subseteq L$ and $M_0 \subseteq M$ such that $L_0$ and $M_0$ 
are isomorphic, and they are \emph{compatible} if there exist 
uncountable suborders $L_0 \subseteq L$ and $M_0 \subseteq M$ such that $L_0$ and $M_0$ 
are either isomorphic or anti-isomorphic. 
Since the existence of a two element basis for the Countryman lines is inconsistent 
with \textsf{CH}, we focus on formally weaker statements 
whose consistency was asked in the early literature on the subject:
\begin{enumerate}
	\item[(A)] (Shelah's conjecture) Every Aronszajn line contains a 
	Countryman line (\cite[Conjecture (A)]{shelahcountryman}).
	\item[(B)] Any two Countryman lines are compatible 
	(\cite[Conjecture (B)]{shelahcountryman}).
	\item[(C)] Any two Aronszajn lines are compatible 
	(\cite{todortreeslinear, baumgartner2}).
\end{enumerate}
The consistency of (B) was proven by {Todor\v{c}evi\'{c}} using the Proper Forcing 
Axiom for posets of size at most $\aleph_1$. 
Assuming that $\textsf{MA}_{\omega_1}$ holds and any two normal 
Aronszajn trees are club isomorphic, if we let $C$ denote any of the Countryman lines 
$C(\rho_0)$, $C(\rho_1)$, $C(\rho_3)$, or $C(\rho)$ defined from 
characteristics of walks on ordinals, 
then $C$ and its converse $C^*$ constitute a two element basis for the Countryman 
lines.\footnote{$\textsf{MA}_{\omega_1}$ is the version of Martin's axiom which asserts 
the existence of a filter on any c.c.c.\ forcing which meets a given family of 
at most $\omega_1$-many dense sets.} 
This two element basis theorem for the Countryman lines is derived from the 
the conjunction of two statements: 
(i) $C$ is minimal, that is, it embeds into all of its uncountable suborders, and 
(ii) every Countryman line is compatible with $C$ 
(see e.g.\ \cite[Cor.\ 2.1.13]{todorbook}).\footnote{As pointed out to 
the first author by Justin Moore, there is an error in \cite{todorbook} whereby it is 
asserted that these statements follow only from $\textsf{MA}_{\omega_1}$. 
Statement (ii) needs the additional assumption about club isomorphisms.} 

The conjunction of (i) and (ii) is inconsistent with \textsf{CH}, but we can ask 
whether versions of (i) or (ii) are consistent with \textsf{CH} individually. 
For (i), Baumgartner \cite[Th.\ 4.15]{baumgartner2} proved that $\Diamond^+$ implies that 
there exists a minimal Aronszajn line. 
Baumgartner's example is a Suslin line (that is, a non-separable 
linear order which 
has no uncountable family of pairwise 
disjoint non-empty open intervals), and hence is not Countryman. 
More recently, Cummings, Eisworth, and Moore \cite{cem} proved that $\Diamond$ 
implies the existence of a minimal Countryman line.
Concerning (ii), Shelah \cite{shelahcountryman} observed that $\Diamond$ implies the existence of 
$2^{\omega_1}$ many pairwise non-compatible Countryman lines. 
In contrast, in this article we prove:

\begin{thm2}
	It is consistent that \textsf{CH} holds 
	and any two Countryman lines are compatible.
\end{thm2}

Note that (C) follows from (A) and (B), so it is natural to approach the consistency 
of (C) by way of Shelah's conjecture. 
The \emph{coloring axiom for trees} (\textsf{CAT}) is the statement 
that for any Aronszajn tree $T$ and for any $K \subseteq T$, there exists an uncountable 
antichain of $T$ such that all meets of pairs of elements in the antichain lie in $K$ or 
all lie in $T \setminus K$. 
When this property holds for a specific Aronszajn tree $T$, we say that \textsf{CAT} 
holds for $T$. 
The coloring axiom for trees was introduced by 
Abraham and Shelah \cite{AS}, who stated (without proof) that, 
under additional hypotheses, \textsf{CAT} is equivalent to Shelah's conjecture. 
Later, a proof of this equivalence 
appeared in \cite[Prop.\ 8.7]{lipschitz}, where it is shown in particular that 
\textsf{CAT} implies Shelah's conjecture under the assumptions:
\begin{enumerate}
	\item $\textsf{MA}_{\omega_1}$;
	\item Any two normal Aronszajn trees are club isomorphic.
\end{enumerate}
Each of these hypotheses implies the failure of \textsf{CH}. 
However, an analysis of the proof of \cite[Prop.\ 8.7]{lipschitz} reveals that, with 
a mild adjustment of the argument, \textsf{CAT} implies Shelah's conjecture 
under the weaker assumptions:
\begin{enumerate}
	\item[($1'$)] Any normal infinitely splitting Aronszajn tree contains a 
	normal downwards closed binary subtree.
	\item[($2'$)] Any two normal Aronszajn trees contain normal downwards closed 
	subtrees which are club isomorphic.
\end{enumerate}
The consistency of ($2'$) with \textsf{CH} was pointed out 
by Abraham and Shelah \cite{AS} (without proof), 
and the consistency of ($1'$) with \textsf{CH} was proven in \cite{jk45}.  
These observations suggest that if \textsf{CH} 
is consistent with \textsf{CAT}, then it is likely that \textsf{CH} is also consistent 
with Shelah's conjecture.

Moore's proof of the consistency of Shelah's conjecture involved showing that 
\textsf{CAT} holds for a particular coherent Aronszajn tree and was inspired in part by 
an earlier and related result of {Todor\v{c}evi\'{c}} (\cite{todorbook}). 
Recall that a tree $T$ is \emph{coherent} if it is a downwards closed subtree of 
${}^{\underaccent{\breve}{\omega}_1} \omega$ 
such that any two elements of the tree agree as functions on 
all but finitely many members of the intersection of their 
domains.\footnote{For any cardinal $\kappa$ and set $X$, 
${}^{\underaccent{\breve}{\kappa}} X$ is the 
collection of all functions of the form 
$f : \alpha \to X$, where $\alpha < \kappa$.} 
For a given coherent Aronszajn tree $T$, define $\Ucal(T)$ to be the collection of sets 
$X \subseteq \omega_1$ for which there exists an uncountable antichain $A \subseteq T$ 
such that for all distinct $x, y \in A$, the height of the meet of $x$ and $y$ is in $X$. 
Assuming that $T$ is \emph{non-Suslin} (that is, any uncountable subset of $T$ 
contains an uncountable antichain), $\Ucal(T)$ is a filter. 
{Todor\v{c}evi\'{c}} \cite[Th.\ 4.1.13]{todorbook} proved that 
$\textsf{MA}_{\omega_1}$ implies that $\Ucal(T)$ is an 
ultrafilter.\footnote{The weaker hypothesis that 
the c.c.c.\ property of posets is productive also implies that $\Ucal(T)$ 
is an ultrafilter. 
This hypothesis is inconsistent with \textsf{CH} (\cite{todorentangled}).}  
The statement that $\Ucal(T)$ is an ultrafilter is a 
weak form of \textsf{CAT} in which we restrict the 
required dichotomy to sets $K \subseteq T$ of the form $T \res X$ 
for some $X \subseteq \omega_1$. 
This connection of $\Ucal(T)$ with Shelah's conjecture motivates the second main theorem 
of the article.

\begin{thm2}
	It is consistent with \textsf{CH} that for any coherent Aronszajn tree 
	$T \subseteq {}^{\underaccent{\breve}{\omega}_1} \omega$, $\Ucal(T)$ is an ultrafilter.
\end{thm2}

A number of open problems are suggested by the above discussion, which we list 
at the end of the article. 
For now we highlight one of the strongest possibilities, which can be thought of 
as an analogue of the two element basis theorem for Aronszajn lines in the 
context of \textsf{CH}.

\begin{probnumber} \label{problem 1}
	Is it consistent with \textsf{CH} that for any Aronszajn lines $L$ and $J$, 
	there exists a Countryman line $C$ such that each of $L$ and $J$ contain 
	either a copy of $C$ or its converse $C^*$?
\end{probnumber}

One scenario in which a positive solution to Problem 1 
could be realized would be if a model of Theorem 1 is combined with 
the family of all uncountable suborders of $C(\rho_3)$ and $C(\rho_3)^*$ 
constituting a basis for the Aronszajn lines.

We give a short overview of the article. 
In Section 2, we provide some background information about trees. 
In Sections 3--9 we prove the first main theorem, and in Sections 10--16 
we prove the second. 
These two parts can be read independently of each other. 
In Section 17, we state further results and some open problems. 

\section{Background on Trees and Linear Orders}

We assume that the reader is familiar with trees and lexicographical orderings of trees, 
but for clarity we review some prominent ideas, facts, and notation which we use. 
For more background see \cite{todortreeslinear}, 
and for notation see \cite[Sec.\ 1]{jk45}. 
If $L$ is a linear order, $L^*$ denotes the \emph{converse} of $L$ which is obtained 
by reversing the order of elements of $L$ (that is, $x <_L y$ iff $y <_{L^*} x$). 
Linear orders $L$ and $M$ are \emph{anti-isomorphic} if $L^*$ and $M$ are isomorphic. 
An uncountable linear order $L$ is a \emph{Countryman line} if $L^2$ is a countable 
union of chains in the componentwise order. 
If $L$ is a Countryman line, then: (1) $L$ is an Aronszajn line, 
(2) $L$ does not contain anti-isomorphic uncountable suborders, and 
(3) for all $2 \le n < \omega$, $L^n$ is a countable 
union of chains (\cite[Th.\ 5.4]{todortreeslinear}). 
Note that any uncountable suborder of a Countryman line is also a Countryman line, and 
if $L$ is a Countryman line then so is $L^*$.

An $\omega_1$-tree (that is, a tree of height $\omega_1$ with countable levels) is 
\emph{normal} if (1) it has a root, (2) every element has incomparable elements above it, 
(3) is Hausdorff (that is, 
distinct elements of the same limit height have different sets of predecessors), and 
(4) every element has uncountably many elements above it. 
In this article, we are concerned almost exclusively with $\omega_1$-trees 
which are uncountable downwards closed subtrees of 
$({}^{\underaccent{\breve}{\omega}_1} \omega,\subsetneq)$. 
For any such tree, (1) and (3) are automatic, and (2) follows from (4) if it is Aronszajn. 
Any $\omega_1$-tree contains an uncountable downwards closed subtree which satisfies (4). 
Hence, any downwards closed Aronszajn subtree of 
${}^{\underaccent{\breve}{\omega}_1} \omega$ contains 
a downwards closed normal subtree.

Suppose that $T$ is an $\omega_1$-tree. 
For any $n < \omega$ and set $X$, $X^n$ 
denotes the $n$-dimensional Cartesian product of $X$. 
Let $T^{\otimes n}$ denote the subset of $T^n$ consisting of tuples whose elements all 
have the same height in $T$. 
The set $T^{\otimes n}$ ordered componentwise is itself a tree. 
Similarly, if $T$ and $S$ are $\omega_1$-trees, then $T \otimes S$ 
is the tree of pairs $(x,y)$, where $x \in T_\alpha$ and $y \in S_\alpha$ for 
some $\alpha < \omega_1$. 
For any $x \in T$, we let $T_x = \{ y \in T : x \le_T y \}$.
An element $x \in T$ is said to \emph{split} if it has at least two immediate successors.
An $\omega_1$-tree is \emph{non-Suslin} if every uncountable subset of it 
contains an uncountable antichain (for Aronszajn trees, 
this property is equivalent to the tree not containing 
a downwards closed Suslin subtree). 
For each $x \in T$ and $\xi \le \h_T(x)$, let $x \res \xi$ denote the unique 
member of $T$ with height $\xi$ such that $x \res \xi \le_T x$. 
If $X \subseteq T_\alpha$ for some $\alpha$ and $\xi \le \alpha$, we say that 
$X$ has \emph{unique drop-downs to $\xi$} if the function which maps any $x \in X$ 
to $x \res \xi$ is injective on $X$.

Assume that $T$ is a downwards closed uncountable subtree of 
${}^{\underaccent{\breve}{\omega}_1} \omega$. 
The height of any element of $T$ is equal to its domain. 
So if $\xi \le \h_T(x)$, then $x \res \xi$ as defined above is the same as the 
restriction of the function $x$ to $\xi$. 
We say that $T$ is \emph{coherent} if any two elements of $T$ agree as functions 
on all but finitely many elements of the intersection of their domain. 
For incomparable $x, y \in T$, let $\Delta(x,y)$ denote the largest ordinal 
$\gamma < \min(\h_T(x),\h_T(y))$ such $x \res \gamma = y \res \gamma$. 
Equivalently, $\gamma$ is the least ordinal such that $x(\gamma) \ne y(\gamma)$, 
and $\gamma$ is the height of the meet $x \wedge y$ of $x$ and $y$. 
For an antichain $A \subseteq T$, let $\Delta (A) = \{ \Delta(x,y) : x \ne y, \ x, y \in A \}$. 
Define $\Ucal(T)$ to be the collection of all sets 
$X \subseteq \omega_1$ such that for some uncountable 
antichain $A \subseteq T$, $\Delta(T) \subseteq X$. 
If $T$ is coherent, Aronszajn, and non-Suslin, then $\Ucal(T)$ is a uniform filter on $\omega_1$.  
The \emph{lexicographical ordering of $T$} is the linear order $L$ 
defined by letting $x <_L y$ if either $x <_T y$, or else $x$ and $y$ are incomparable 
and $x(\Delta(x,y)) < y(\Delta(x,y))$. 
If $T$ is Aronszajn, then the lexicographical ordering of $T$ is an Aronszajn line 
(for a converse, see the beginning of Section \ref{Compatibility of Countryman Lines} below). 
Note that if $x$ and $y$ are incomparable, $x <_T x^+$, and $y <_T y^+$, 
then $\Delta(x,y) = \Delta(x^+,y^+)$. 
This fact easily implies the following useful lemma which we use frequently.

\begin{lemma} \label{easy lex lemma}
	Let $T$ be a downwards closed subtree of 
	${}^{\underaccent{\breve}{\omega}_1} \omega$ 
	with lexicographical ordering $L$. 
	If $a_0 <_T a_1$, $b_0 <_T b_1$, and $a_0$ and $b_0$ are incomparable in $T$, 
	then $a_0 <_L b_0$ iff $a_1 <_L b_1$. 
	In particular, for all incomparable $a, b \in T$, 
	for all $\Delta(a,b) < \xi \le \h_T(a)$, and 
	for all $\Delta(a,b) < \gamma \le \h_T(b)$, 
	$a <_L b$ iff $a \res \xi <_L b \res \gamma$.
\end{lemma}

The following lemma has a straightforward proof.

\begin{lemma} \label{club lemma}
	Suppose that $T$ is a normal Aronszajn tree which is a downwards closed 
	subtree of ${}^{\underaccent{\breve}{\omega}_1} \omega$. 
	Let $L$ be the lexicographical ordering of $T$. 
	Then there exists a club $E \subseteq \omega_1$ such that for all $\eta \in E$ 
	and for all $x \in T \res \eta$, the set of successors of $x$ in $T_\eta$, 
	considered as a suborder of $L$, is isomorphic to the rationals.
\end{lemma}

\section{Compatibility of Countryman Lines} \label{Compatibility of Countryman Lines}

In this section, we give an overview of our approach for proving the 
consistency of the statement that \textsf{CH} holds and 
any two Countryman lines are compatible.

We begin by reviewing one way of representing an Aronszajn line as 
a lexicographically ordered Aronszajn tree. 
Suppose that $L$ is an Aronszajn line. 
Then there exists an Aronszajn tree $T$ satisfying the following properties.
	\begin{enumerate}
		\item The elements of $T$ of successor height are intervals 
		of $L$ which have a right endpoint.
		\item The elements of $T$ of limit height are non-empty 
		intervals of $L$.
		\item $T$ has a root which is $L$.
		\item $I <_T J$ iff $J \subseteq I$, for all $I, J \in T$.
		\item Incomparable elements of $T$ are disjoint.
		\item $T$ is Hausdorff.
		\item Every element of $T$ with limit height has an immediate successor.
		\item Linearly order the set of immediate successors of any element of $T$ 
		by letting $I$ be below $J$ iff $\max(I) <_L \max(J)$, and let $M$ denote 
		the corresponding lexicographical ordering of $T$. Then:
		\begin{enumerate}
			\item The set of immediate successors of any element of $T$ is either empty, 
			a singleton, or is isomorphic to $(\omega,\in)$.
			\item The suborder of $M$ consisting of elements of $T$ with successor height 
			is isomorphic to $L$.
		\end{enumerate}
	\end{enumerate}

The proof of the above is a slight modification of the augmented version of the tree 
construction of Theorem 4.2 mentioned in the comments before Theorem 4.7 of \cite{baumgartner2} 
(we use right endpoints instead of left endpoints to obtain (8a)). 
Note that any lexicographically ordered 
tree satisfying the above properties is isomorphic to a downwards closed 
lexicographically ordered subtree of 
${}^{\underaccent{\breve}{\omega}_1} \omega$, so going forward 
we only consider trees of this kind.

\begin{lemma} \label{augmented tree is Countryman}
	Suppose that $T$ is a downwards closed subtree of 
	${}^{\underaccent{\breve}{\omega}_1} \omega$ 
	which is special, every element of limit height has an immediate successor, 
	and the lexicographical ordering of $T$ 
	restricted to its successor levels is a Countryman line. 
	Then the lexicographical ordering of $T$ is a Countryman line.
\end{lemma}

\begin{proof}
	Let $L$ denote the lexicographical ordering of $T$ and let $C$ denote 
	the suborder of $L$ consisting of elements of $T$ with successor height. 
	Fix a specializing function $f : T \to \omega$ and fix a function 
	$g : C^2 \to \omega$ such that for all $n < \omega$, $g^{-1}(n)$ is a chain in $C^2$.  
	For each $c \in T$ not of successor height, 
	pick an arbitrary immediate successor $c^+$ of $c$, and for $c \in T$ 
	of successor height, let $c^+ = c$.  
	Define $h : L^2 \to \omega^3$ by $h(c,d)= (f(c),f(d),g(c^+,d^+))$. 
	Using Lemma \ref{easy lex lemma}, it is easy to check that for all 
	$x \in \omega^3$, $h^{-1}(x)$ is a chain in $L^2$.
\end{proof}

\begin{lemma} \label{isomorphic antichains}
	Let $T$ be an Aronszajn tree which is a downwards closed 
	subtree of ${}^{\underaccent{\breve}{\omega}_1} \omega$ 
	and let $L$ be the lexicographical ordering of $T$. 
	Suppose that $X$ is a dense subset of $T$. 
	Then any uncountable suborder of $L$ which is an antichain of $T$ 
	is isomorphic to an uncountable suborder of $X$.
\end{lemma}

\begin{proof}
	Let $A \subseteq T$ be an uncountable antichain. 
	For each $a \in A$, pick some $x_a \in X$ such that $a \le_T x_a$. 
	Since $A$ is an antichain, Lemma \ref{easy lex lemma} implies that for all 
	distinct $a, b \in A$, $a <_L b$ iff $x_a <_L x_b$.
\end{proof}

The next lemma describes our strategy for proving the first main theorem of the article.

\begin{lemma}
	Assume that all Aronszajn trees are special. 
	Suppose that whenever $L$ and $J$ are Countryman lexicographical 
	orderings of normal downwards closed subtrees 
	$T$ and $S$ of ${}^{\underaccent{\breve}{\omega}_1} \omega$, respectively, 
	then $L$ and $J$ are compatible. 
	Then any two Countryman lines are compatible.
\end{lemma}

\begin{proof}
	Let $X$ and $Y$ be Countryman lines. 
	Fix downwards closed subtrees $T_0$ and $S_0$ of 
	${}^{\underaccent{\breve}{\omega}_1} \omega$ 
	satisfying properties (1)--(8) listed at the beginning of the section such that 
	the lexicographical ordering of the successor levels of $T_0$ and $S_0$ 
	are isomorphic to $X$ and $Y$, respectively. 
	Let $L_0$ and $J_0$ be the lexicographical orderings of $T_0$ and $S_0$, respectively. 
	By Lemma \ref{augmented tree is Countryman}, $L_0$ and $J_0$ are Countryman lines. 
	Fix normal subtrees $T$ and $S$ of $T_0$ and $S_0$, respectively, and let 
	$L$ and $J$ denote the lexicographical orderings of $T$ and $S$, respectively. 
	Since $L \subseteq L_0$ and $J \subseteq J_0$, 
	$L$ and $J$ are also Countryman lines.

	By the assumptions of the lemma, $L$ and $J$ are compatible. 
	So fix uncountable sets $A_0 \subseteq L$ and $B_0 \subseteq J$ 
	such that $A_0$ and $B_0$ are either isomorphic or anti-isomorphic as 
	witnessed by an order preserving or order reversing bijection $G : A_0 \to B_0$. 
	Applying the fact that $A_0 \subseteq T$ and $T$ is special, fix an uncountable 
	set $A_1 \subseteq A_0$ which is an antichain of $T$. 
	Similarly, as $S$ is special fix an uncountable set $B \subseteq G[A_1]$ 
	which is an antichain of $S$. 
	Let $A = G^{-1}[B]$. 
	Then $A$ and $B$ are either isomorphic or anti-isomorphic. 
	Now Lemma \ref{isomorphic antichains} applied to each of $T_0$ and $S_0$ implies that 
	$A$ is isomorphic to a subset of $X$ and $B$ is isomorphic to a subset of $Y$. 
	Therefore, $X$ and $Y$ contain uncountable subsets which are either 
	isomorphic or anti-isomorphic.
\end{proof}

We work with normal special Aronszajn trees $T$ and $S$ which are downwards 
closed subtrees of 
${}^{\underaccent{\breve}{\omega}_1} \omega$ with corresponding lexicographical orderings 
$L$ and $J$ which are each Countryman lines. 
We consider two possibilities. 
The first is that $L^*$ and $J$ are near. 
In that case, $T$ and $S$ are already handled. 
Assume that $L^*$ and $J$ are not near. 
Then our goal becomes to design a forcing $\p(T,S)$ 
which adds a club set $C \subseteq \omega_1$ 
and a function $F : T \res C \to S \res C$ satisfying:
\begin{itemize}
	\item $F$ is level preserving (that is, for all $x \in T \res C$, 
	$\h_S(F(x)) = \h_T(x)$);
	\item for all $x$ and $y$ in $T \res C$:
	\begin{itemize}
		\item $x <_T y$ implies $F(x) <_S F(y)$;
		\item $x <_L y$ implies $F(x) <_J F(y)$.
	\end{itemize}
\end{itemize}

The forcing $\p(T,S)$ adds the function as above using countable approximations. 
Specifically, conditions include a \emph{working part} which is a function 
$f : c_f \to S$, where $c_f$ is a closed countable subset of $\omega_1$ and $f$ 
satisfies the description of $F$ above for members of its domain. 
There are two main types of properties which we prove of this forcing. 
The first type is \emph{extension}, which means extending the working part 
of a condition arbitrarily high in the tree while satisfying various requirements. 
The second type is the ability to build 
generic conditions over countable elementary substructures 
to prove properness and related properties. 
For this second purpose, the conditions in our forcing involve a 
second component, which is a countable 
set of subtrees of product trees called \emph{promises} which put additional 
restrictions on the working part. 
The method of promises was introduced by Shelah in his forcing for specializing 
an Aronszajn tree while not adding reals (\cite[Ch.\ V, Sec. 6]{shelahbook}). 

In order for $\p(T,S)$ to preserve \textsf{CH}, 
we prove that it is \emph{totally proper}, which means that it is 
proper and does not add reals. 
Equivalently, for any appropriate countable elementary substructure $N$, every condition 
in $N$ has an extension $q$ which is a \emph{total master condition}, meaning that for every 
dense set $D \subseteq \p(T,S)$ in $N$, there is some condition in $N \cap D$ above $q$. 
A general approach for proving total properness is to enumerate all 
dense sets in $N$ as $\langle D_n : n < \omega \rangle$, 
build a descending sequence $\langle p_n : n < \omega \rangle$ of conditions 
such that for each $n$, $p_{n+1} \in D_n \cap N$, and 
then define a lower bound of this sequence. 
A minor adjustment in this argument shows that $\p(T,S)$ satisfies a stronger property 
called \emph{$(<\! \omega_1)$-proper}, 
which means that total master conditions exists not only for a 
single countable model, but for all models appearing in an $\in$-chain 
of countable models of any countable length. 
Another adjustment shows that $\p(T,S)$ has the 
$\omega_2$-p.i.c. (\cite[Ch.\ VIII, Def.\ 2.1]{shelahbook}).

In order to obtain the desired consistency result, we need to iterate the above 
forcing so that the iteration itself does 
not add reals or collapse $\omega_2$. 
Working over a model of $2^\omega = \omega_1$ and $2^{\omega_1} = \omega_2$, 
we use a countable support iteration of length $\omega_2$ of forcings of the type 
above together with forcings which specialize Aronszajn trees without adding reals. 
The fact that the forcing iteration is totally proper is a consequence of 
the following iteration theorem of Shelah: 
Any countable support forcing iteration of forcings which are $(<\! \omega_1)$-proper 
and Dee-complete is totally proper (\cite[Ch.\ V, Th.\ 7.1]{shelahbook}). 
Finally, for the preservation of cardinals greater than $\omega_1$ as well as for 
the usual bookkeeping process, the forcings we use 
satisfy the $\omega_2$-p.i.c.
This property is a strong form of the $\omega_2$-chain condition which is preserved, 
over a model of \textsf{CH}, by any countable support forcing iteration of length 
less than $\omega_2$, and guarantees that such an iteration of length $\omega_2$ is 
$\omega_2$-c.c.\ (\cite[Ch.\ VIII, Lem.\ 2.4]{shelahbook}). 
Finally, in order to ensure a successful bookkeeping of all trees which appear in the final 
model, we use the fact that our forcings have size $2^{\omega_1} = \omega_2$ 
and the iteration is $\omega_2$-c.c.

To summarize, in order to prove the first consistency result we need forcings of the 
form $\p(T,S)$ to be $(<\! \omega_1)$-proper, Dee-complete, and $\omega_2$-p.i.c. 
We also use the fact due to Shelah \cite{shelahbook} that any Aronszajn tree 
can be made special using a forcing which has these properties as well. 
But the core piece of information about the forcing $\p(T,S)$ is that it is totally proper. 
The verification of $(<\! \omega_1)$-proper is a slight adjustment of the proof of this fact, 
and the proofs of Dee-completeness and $\omega_2$-p.i.c.\ are routine and 
duplicate proofs of these properties for similar posets in the literature. 
To avoid artificially extending the length and complexity of the article by including 
routine but lengthy details, 
we only verify total properness and leave the interested reader to explore 
the other properties using arguments from other sources (specifically, see 
\cite[Ch.\ V, Th.\ 6.1]{shelahbook} 
and \cite[Sec.\ 10]{jk45}).\footnote{Complete proofs of all results 
stated in this article 
are included in the second author's PhD dissertation.}

For the remainder of the proof of Theorem 1, 
fix normal special Aronszajn trees $T$ and $S$ which are downwards 
closed subtrees of ${}^{\underaccent{\breve}{\omega}_1} \omega$ 
with corresponding lexicographical orderings $L$ and $J$ which are each 
Countryman lines. 
Applying Lemma \ref{club lemma} to $S$, 
fix a club $E \subseteq \omega_1$ such that 
for all $\eta \in E$ and for all $x \in S \res \eta$, 
the set of successors of $x$ in $S_\eta$, 
considered as a suborder of $J$, is isomorphic to the rationals. 
Most of the properties of the poset $\p(T,S)$ and its associated objects 
can be proved without additional hypotheses, 
but at several key places we assume that $L^*$ and $J$ are not near.

\section{Tuples, 1}

We work with finite tuples of pairs of the form 
$\vec x = ((a_k,b_k) : k < n)$, 
where $n < \omega$ and for some $\alpha \in E$, 
$a_0,\ldots,a_{n-1}$ are in $T_\alpha$ 
and $b_0,\ldots,b_{n-1}$ are in $S_{\alpha}$. 
For a fixed $n$, we consider such tuples as elements of the $n$-dimensional product tree 
$((T \res E) \otimes (S \res E))^{\otimes n}$. 

\begin{defn}
	Let $0 < n < \omega$, $\alpha \in E$, and 
	$\vec x = ((a_k,b_k) : k < n) \in 
	(T_\alpha \times S_{\alpha})^{n}$. 
	\begin{enumerate}
		\item We refer to $\alpha$ as the \emph{height} of $\vec x$.
		\item $\vec x$ is \emph{lex-increasing} if 
		$a_0 <_L \cdots <_L a_{n-1}$ and 
		$b_{0} <_J \cdots <_J b_{n-1}$.
		\item For any $\xi \le \alpha$:
		\begin{enumerate}
			\item $\vec x$ has 
			\emph{unique drop-downs to $\xi$} if 
			$(a_k \res \xi : k < n)$ and 
			$(b_k \res \xi : k < n)$ 
			are each injective;
			\item Define 
			$\vec x \res \xi = 
			((a_k \res \xi,b_k \res \xi) : k < n)$.
		\end{enumerate}
	\end{enumerate}
\end{defn}

\begin{lemma} \label{lex increasing drop 1}
	Let $\xi < \alpha \in E$. 
	Suppose that $\vec x \in (T_\alpha \otimes S_{\alpha})^n$ and $\vec x$ 
	has unique drop-downs to $\xi$. 
	Then $\vec x$ is lex-increasing iff $\vec x \res \xi$ is lex-increasing.
\end{lemma}

\begin{defn}
	Consider $\alpha \in E$ and 
	lex-increasing tuples 
	$\vec x = ((a_k,b_k) : k < m)$ and $\vec y = ((c_k,d_k) : k < n)$ 
	in $(T_\alpha \times S_{\alpha})^{<\omega}$. 
	\begin{enumerate}
		\item The tuples $\vec x$ and $\vec y$ are \emph{disjoint} if 
		$a_i \ne c_j$ and $b_i \ne d_j$ for all $i < m$ and $j < n$. 
		\item If $\vec x$ and $\vec y$ are disjoint, we say that they are 
		\emph{compatible} if for all $i < m$ and $j < n$, 
		$a_i <_L c_j$ iff $b_i <_J d_j$.
	\end{enumerate}
\end{defn}

An easy fact which is used frequently below is that 
if $\vec x$ and $\vec y$ are compatible, then there exists a lex-increasing tuple 
which lists exactly the pairs occurring in either $\vec x$ or $\vec y$.

\begin{lemma} \label{compatible drop 1}
	Let $\xi < \alpha \in E$. 
	Suppose that $\vec x$ and $\vec y$ are lex-increasing tuples in 
	$(T_\alpha \otimes S_{\alpha})^{<\omega}$ each of which has 
	unique drop-downs to $\xi$ and satisfy that 
	$\vec x \res \xi$ and $\vec y \res \xi$ are disjoint. 
	Then $\vec x$ and $\vec y$ are compatible iff $\vec x \res \xi$ and 
	$\vec y \res \xi$ are compatible.
\end{lemma}

Lemmas \ref{lex increasing drop 1} and \ref{compatible drop 1} 
both follow easily from Lemma \ref{easy lex lemma}.

Suppose that 
$\vec x = ((a_k,b_k) : k < n)$ and $\vec y = ((c_k,d_k) : k < n)$ 
are in $((T \res E) \otimes (S \res E))^{\otimes n}$, not necessarily of the same height. 
We say that $\vec x$ and $\vec y$ are \emph{componentwise incomparable} if 
for all $k < n$, $a_k$ and $c_k$ are incomparable in $T$ and 
$b_k$ and $d_k$ are incomparable in $S$. 
We say that $\vec x$ and $\vec y$ are \emph{fully incomparable} if for all $k, m < n$, 
$a_k$ and $c_m$ are incomparable in $T$ and $b_k$ and $d_m$ are incomparable in $S$. 

\section{Promises, 1}

We use a version of the method of promises of Shelah \cite[Ch.\ V, Sec.\ 6]{shelahbook}. 
The type of promises which we use for the proof of Theorem 1 is described next.

\begin{defn}[Promises, 1]
	Define $\Pcal$ to be the set of all $P$ satisfying that for some 
	$0 < n < \omega$ and $\delta \in E$:
	\begin{enumerate}
		\item $P$ is a non-empty downwards closed subset of the tree 
		$[(T \res (E \setminus \delta)) \otimes 
		(S \res (E \setminus \delta))]^{\otimes n}$;
		\item for all $\vec x \in P$, $\vec x$ is lex-increasing;
		\item for all $\vec x \in P$, 
		there are uncountably many elements of $P$ above $\vec x$.
	\end{enumerate}
	In the above, we refer to $\delta$ as the \emph{base level of $P$}, 
	which we denote by $\delta_P$. 
	The set of $P \in \Pcal$ with base level $\delta$ is denoted by $\Pcal_\delta$.
\end{defn}

Following Shelah, 
we informally refer to the members of $\Pcal$ as \emph{promises}.

\begin{prop} \label{promise existence start 1}
	Assume that $L^*$ and $J$ are not near. 
	For all $P \in \Pcal$ and for any $\vec x \in P$, there exists a family 
	$\{ \vec x_n : n < \omega \}$ of elements of $P$ of the same height above $\vec x$ 
	which is pairwise disjoint and pairwise compatible.
\end{prop}

\begin{proof}
	Write $\vec x = ((a_k,b_k) : k < n)$. 
	For each $\alpha \in E$ greater than the height of $\vec x$, 
	pick some $\vec x_\alpha \in P$ 
	with height $\alpha$ above $\vec x$. 
	Write each $\vec x_\alpha$ as $((a_k^\alpha,b_k^\alpha) : k < n)$. 
	Since $\vec x$ is injective and 
	$a_k <_T a_k^\alpha$ and $b_k <_T b_k^\alpha$ for all $\alpha$ and $k < n$, 
	it follows that for all distinct $k, l < n$ and for all $\alpha \le \beta$, 
	$a_k^\alpha$ and $a_l^\beta$ are incomparable in $T$ and 
	$b_k^\alpha$ and $b_l^\beta$ are incomparable in $S$. 

	Applying the fact that $T$ and $S$ are special in $2n$-many steps, 
	fix an uncountable set $X \subseteq E$ of ordinals greater than the height of $\vec x$ 
	so that for all $\alpha < \beta$ in $X$, $\vec x_\alpha$ and $\vec x_\beta$ 
	are componentwise incomparable. 
	In particular, for all $\alpha < \beta$ in $X$, $\vec x_\alpha$ and $\vec x_\beta$ 
	are fully incomparable and hence $\vec x_\alpha$ and 
	$\vec x_\beta \res \alpha$ are disjoint. 
	Using the fact that $L$ and $J$ are Countryman lines in two steps, 
	fix an uncountable set $Y \subseteq X$ such that 
	$\{ (a_0^\alpha,\ldots,a_{n-1}^\alpha) : \alpha \in Y \}$ is 
	a chain in $L^n$ and 
	$\{ (b_0^\alpha,\ldots,b_{n-1}^\alpha) : \alpha \in Y \}$ 
	is a chain in $J^n$.

	Define a function $F : [Y]^2 \to 2$ so that for all $\alpha < \beta$ in $Y$, 
	$F(\{ \alpha, \beta \})$ is equal to $1$ if 
	$\vec x_\alpha$ and $\vec x_\beta \res \alpha$ are compatible, and $0$ otherwise. 
	By the Dushnik-Miller theorem, there exists a set $Z \subseteq Y$ such that 
	either (a) $Z$ is infinite and for all $\alpha < \beta$ in $Z$, 
	$F(\{ \alpha, \beta \}) = 1$, 
	or else (b) $Z$ is uncountable and for all $\alpha < \beta$ in $Z$, 
	$F(\{ \alpha,\beta \}) = 0$.

	Suppose for a contradiction that (b) holds. 
	Consider any $\alpha < \beta$ in $Z$. 
	Since $F(\{ \alpha,\beta \}) = 0$, $\vec x_\alpha$ and $\vec x_\beta \res \alpha$ 
	are not compatible. 
	By Lemma \ref{easy lex lemma} and the fact that $\vec x$ is lex-increasing, 
	for any distinct $k, l < n$, 
	$a_k^\alpha <_L a_l^\beta$ iff $a_k <_L a_l$ iff $k < l$ iff 
	$b_k <_J b_l$ iff $b_k^\alpha <_J b_l^\beta$. 
	Consequently, there exists some $k < n$ such that either 
	$a_k^\alpha <_L a_k^\beta$ and $b_k^\alpha >_J b_k^\beta$, or 
	$a_k^\alpha >_L a_k^\beta$ and $b_k^\alpha <_J b_k^\beta$. 
	Since $\{ (a_0^\xi,\ldots,a_{n-1}^\xi) : \xi \in Z \}$ and 
	$\{ (b_0^\xi,\ldots,b_{n-1}^\xi) : \xi \in Z \}$ are each chains, 
	the last sentence is true for $0$ in place of $k$. 
	Therefore, $a_0^\alpha <_L a_0^\beta$ implies $b_0^\alpha >_J b_0^\beta$ 
	and $a_0^\alpha >_L a_0^\beta$ implies $b_0^\alpha <_J b_0^\beta$. 
	So the map $g : \{ a_0^\alpha : \alpha \in Z \} \to \{ b_0^\alpha : \alpha \in Z \}$ 
	defined by $g(a_0^\alpha) = b_0^\alpha$ for all $\alpha \in Z$ 
	is an order preserving function from an uncountable 
	subset of $L^*$ into $J$. 
	Hence $L^*$ and $J$ are near, which is a contradiction.
	
	So (a) holds. 
	Pick some $\delta \in E$ greater than the height of 
	$\vec x_\alpha$ for all $\alpha \in Z$. 
	For each $\alpha \in Z$, choose some $\vec x_\alpha^+ \in P$ 
	with height $\delta$ which is above $\vec x_\alpha$. 
	Now for all $\alpha < \beta$ in $Z$, $\vec x_\alpha$ and 
	$\vec x_\beta \res \alpha$ are disjoint and compatible. 
	By Lemma \ref{compatible drop 1}, $\vec x_\alpha^+$ and $\vec x_\beta^+$ 
	are disjoint and compatible.
\end{proof}

\section{The Working Part, 1}

The forcing $\p(T,S)$ adds the desired club and function by countable approximations. 
The next definition describes these objects.

\begin{defn}
	Define $\Fcal$ to be the set of functions of the form 
	$f : T \res c_f \to S \res c_f$, 
	where $c_f$ is a closed countable subset of $E$, satisfying:
	\begin{enumerate}
		\item $f$ is level preserving (that is, for all $x \in \dom(f)$, 
		$\h_T(x) = \h_S(f(x))$;
		\item for all $x, y \in \dom(f)$:
		\begin{enumerate}
			\item $x <_T y$ implies $x <_S y$;
			\item $x <_L y$ implies $f(x) <_J f(y)$.
		\end{enumerate}
	\end{enumerate}
	In the above, let $\alpha_f$ denote the maximum element of $c_f$.
\end{defn}

\begin{defn}
	Suppose that $0 < n < \omega$, $\xi < \omega_1$, 
	and $\vec x = ((a_k,b_k) : k < n)$ is in $(T_\xi \times S_{\xi})^n$. 
	For any $f \in \Fcal$ with $\xi \in c_f$, 
	we say that $f$ \emph{is consistent with $\vec x$} if for all $k < n$, $f(a_k) = b_k$.
\end{defn}

Note that if $f \in \Fcal$ is consistent with two disjoint tuples of the same height, 
then the tuples are compatible.

\begin{lemma} \label{compatible drop 1 bonus}
	Assume the following:
	\begin{enumerate}
		\item $f \in \Fcal$;
		\item $\alpha_f \le \delta \in E$;
		\item $\vec x$ and $\vec y$ are lex-increasing tuples in 
		$(T_\delta \times S_\delta)^{<\omega}$ each with unique drop-downs to 
		$\alpha_f$;
		\item $\vec x \res \alpha_f$ and $\vec y \res \alpha_f$ are disjoint;
		\item $f$ is consistent with $\vec x \res \alpha_f$ and $\vec y \res \alpha_f$.
	\end{enumerate}
	Then $\vec x$ and $\vec y$ are compatible. 
	Hence, there is a lex-increasing 
	tuple in $(T_\delta \times S_\delta)^{<\omega}$ which lists the elements 
	of $\vec x$ and $\vec y$.
\end{lemma}

The proof follows easily from Lemma \ref{compatible drop 1}.

\begin{defn}
	Suppose that $f \in \Fcal$ and $P \in \Pcal$ has dimension $n$. 
	We say that $f$ \emph{fulfills $P$} if $\delta_P \le \alpha_f$ and 
	there exists 
	a pairwise disjoint family $\{ \vec x_n : n < \omega \}$ of 
	members of $P$ with height $\alpha_f$ such that for all $n < \omega$, 
	$f$ is consistent with $\vec x_n$.
\end{defn}

The next lemma assists with constructing members of $\Fcal$ in limit processes.

\begin{lemma} \label{assisting lemma}
	Assume the following:
	\begin{enumerate}
		\item $\langle \alpha_n : n < \omega \rangle$ is an increasing sequence 
		of ordinals in $E$ with supremum $\delta$;
		\item for each $n$:
		\begin{enumerate}
			\item $f_n \in \Fcal$ and $\alpha_n = \alpha_{f_n}$;
			\item $f_n \subseteq f_{n+1}$;
		\end{enumerate}
		\item $g : T \res c_g \to S \res c_g$, where 
		$c_g = (\bigcup_n c_{f_n}) \cup \{ \delta \}$;
		\item $g \res (T \res \delta) = \bigcup_n f_n$;
		\item $g \res T_\delta$ is a lexicographical-order preserving function mapping 
		into $S_{\delta}$;
		\item for all $x \in T_\delta$, there exists $n$ such that for all $m \ge n$, 
		$f_m(x \res \alpha_m) = g(x) \res \alpha_m$.
	\end{enumerate}
	Then $g \in \Fcal$.
\end{lemma}

\begin{proof}
	Clearly, $g$ is level preserving. 
	Suppose that $c <_T d$ are in $\dom(g)$ and we show that 
	$g(c) <_S g(d)$. 
	This is obvious except when $d \in T_\delta$. 
	Fix $m$ large enough so that $\h_T(c) < \alpha_m$ and 
	$f_m(d \res \alpha_m) = g(d) \res \alpha_m$. 
	Then $g(c) = f_m(c) <_S f_m(d \res \alpha_m) <_S g(d)$.

	Now suppose that $c$ and $d$ are in $\dom(g)$ and $c <_L d$. 
	We prove that $g(c) <_J g(d)$. 
	This is immediate if $c, d \in T_\delta$ or if $c, d \in T \res \delta$. 
	If $c <_T d$, then by the above, $g(c) <_S g(d)$, so $g(c) <_J g(d)$. 
	Otherwise, $c$ and $d$ are incomparable in $T$.
	
	\emph{Case 1:} $c \in T \res \delta$ and $d \in T_\delta$. 
	Pick $m$ large enough so that 
	$\alpha_m > \h_T(c)$ and $f_m(d \res \alpha_m) = g(d) \res \alpha_m$. 
	By Lemma \ref{easy lex lemma}, $c <_L d$ implies that 
	$c <_L d \res \alpha_m$. 
	So $g(c) = f_m(c) <_J f_m(d \res \alpha_m)$, and 
	$f_m(d \res \alpha_m) <_S g(d)$ implies that 
	$f_m(d \res \alpha_m) <_J g(d)$.
	
	\emph{Case 2:} $c \in T_\delta$ and $d \in T \res \delta$. 
	Fix $m$ large enough so that $\alpha_m > \h_T(d)$ and 
	$f_m(c \res \alpha_m) = g(c) \res \alpha_m$. 
	Since $c$ and $d$ are incomparable, 
	by Lemma \ref{easy lex lemma} we have that $c \res \alpha_m <_L d$. 
	Hence, $f_m(c \res \alpha_m) <_J f_m(d)$. 
	In particular, $f_m(d) \not <_S f_m(c \res \alpha_m)$, so 
	$f_m(d)$ and $f_m(c \res \alpha_m)$ are incomparable in $S$. 
	By Lemma \ref{easy lex lemma} and the fact that $f_m(c \res \alpha_m) <_S g(c)$, 
	$g(c) <_J f_m(d) = g(d)$.
\end{proof}

\section{Extension, 1} \label{Extension, 1}

The next lemma allows us to stretch a condition up arbitrarily high in the trees. 
Since the generic function is only defined on a club of levels, such simple extensions do not 
require a limit process of the kind described in Lemma \ref{assisting lemma}. 
But we can only avoid the limit case temporarily, and we will confront it in 
Proposition \ref{adding promises 1} and Theorem \ref{totally proper 1} below.

\begin{lemma}[Extension, 1] \label{extension 1}
	Assume the following:
	\begin{enumerate}
		\item $f \in \Fcal$;
		\item $\Gamma \subseteq \Pcal$ is countable and $f$ fulfills every member of $\Gamma$;
		\item $\alpha_f < \delta \in E$;
		\item $\vec x$, $\vec y$, and $\vec z$ are in 
		$(T_{\delta} \times S_{\delta})^{<\omega}$, each with 
		unique drop-downs to $\alpha_f$;
		\item either $\vec y = \vec z = \emptyset$ or $\vec y$ and $\vec z$ 
		are disjoint and compatible;
		\item $\vec y \res \alpha_f = \vec z \res \alpha_f$;
		\item $\vec x \res \alpha_f$ is disjoint from $\vec y \res \alpha_f$;
		\item $f$ is consistent with $\vec x \res \alpha_f$ and $\vec y \res \alpha_f$.
	\end{enumerate}
	Then there exist $g \in \Fcal$ such that:
	\begin{enumerate}
		\item[(a)] $f \subseteq g$;
		\item[(b)] $c_g = c_f \cup \{ \delta \}$;
		\item[(b)] $g$ fulfills every member of $\Gamma$;
		\item[(c)] $g$ is consistent with $\vec x$, $\vec y$, and $\vec z$.
	\end{enumerate}
\end{lemma}

\begin{proof}
	Write $\vec y = ((a^0_k,b^0_k) : k < n)$ and 
	$\vec z = ((a^1_k,b^1_k) : k < n)$. 
	Let $\langle P_m : m < \omega \rangle$ be an enumeration of $\Gamma$ 
	such that every member appears infinitely often. 
	Since $f$ fulfills every member of $\Gamma$, 
	we can fix a pairwise disjoint sequence 
	$\langle \vec w_m : m < \omega \rangle$ satisfying that for all $m < \omega$:
	\begin{enumerate}
		\item $\vec w_m$ is in $P_m$ and has height $\alpha_f$;
		\item $\vec w_m$ is disjoint from $\vec x \res \alpha_f$ and 
		$\vec y \res \alpha_f$;
		\item $f$ is consistent with $\vec w_m$.
	\end{enumerate}
	For each $m < \omega$, fix $\vec w_m^+$ in $P_m$ of height $\delta$ 
	which is above $\vec w_m$. 
	It suffices to define a function $g \in \Fcal_{\delta}$ with $c_g = c_f \cup \{ \delta \}$ 
	such that $f \subseteq g$ and 
	$g$ is consistent with $\vec x$, 
	$\vec y$, $\vec z$, and $\vec w_m$ for all $m < \omega$.

	Consider $x \in T_{\alpha_f}$ and we define a function 
	$g_x$ from the set of the successors of $x$ in $T_\delta$ into 
	the successors of $f(x)$ in $S_{\delta}$. 
	Then we let $h = g \cup \{ g_x : x \in T_{\alpha_f} \}$. 
	We split into three mutually exclusive cases. 

	\emph{Case 1:} There exists a tuple 
	$((a_k,b_k) : k < p)$ in the set 
	$\{ \vec w_m^+ : m < \omega \} \cup \{ \vec x \}$ and 
	there exists $k < p$ such that $x <_T a_k$. 
	Since $f$ is consistent with this tuple's restriction to $\alpha_f$, 
	$x <_T a_k$ implies $f(x) <_S b_k$. 
	As $\delta \in E$, the set of successors of $f(x)$ in $S_{\delta}$, 
	considered as a suborder of $J$, is isomorphic to the rationals. 
	So we can fix a lexicographical-order preserving function 
	$g_x$ from the set of successors of $x$ in $T_{\delta}$ 
	into the set of successors of $f(x)$ in $S_{\delta}$ 
	such that $g_x(a_k) = b_k$.

	\emph{Case 2:} There exists $l < n$ such that $x <_T a^0_l$. 
	Then $x <_T a^1_l$, $f(x) <_S b^0_l$, and $f(x) <_S b^1_l$. 
	Let $Q$ be the set of successors of $f(x)$ in $S_\delta$. 
	Then again $Q$ is isomorphic to the rationals. 
	Since $\vec y$ and $\vec z$ 
	are compatible, $a^0_l <_L a^1_l$ iff $b^0_l <_J b^1_l$. 
	So we can fix a lexicographical-order preserving function 
	$g_x$ from the set of successors of $x$ in $T_{\delta}$ into $Q$ 
	such that $g_x(a^0_l) = b^0_l$ and $g_x(a^1_l) = b^1_l$.

	\emph{Case 3:} Neither case 1 nor case 2 hold. 
	Let $Q$ be the set of successors of $f(x)$ in $S_\delta$. 
	Fix a lexicographical-order preserving function $g_x$ from the set the successors of 
	$x$ in $T_\delta$ into the set of successors of 
	$f(x)$ in $S_{\delta}$.
	
	To show that $g \in \Fcal_\delta$, it suffices to show that $g$ is 
	lexicographical-order preserving. 
	Suppose that $c <_L d$ are in the domain of $g$ and 
	we prove that $g(c) <_J g(d)$. 
	This is immediate if either $c$ and $d$ are in $T \res (\alpha_f + 1)$, or if 
	$c$ and $d$ are in $T_\delta$ and $c \res \alpha_f = d \res \alpha_f$. 
	We consider the remaining possibilities.

	\emph{Case a:} $c$ and $d$ are in $T_{\delta}$ and $c \res \alpha_f \ne d \res \alpha_f$. 
	By Lemma \ref{easy lex lemma} and the fact that $f$ is lexicographical-order preserving, 
	\begin{multline*}
	c <_L d \Longleftrightarrow c \res \alpha_f <_L d \res \alpha_f \Longleftrightarrow \\ 
	g(c) \res \alpha_f = f(c \res \alpha_f) <_J f(d \res \alpha_f) = 
	g(d) \res \alpha_f 
	\Longleftrightarrow g(c) <_J g(d).
	\end{multline*}
	
	\emph{Case b:} $c \in T \res \delta$ and $d \in T_{\delta}$. 
	First, assume that $c <_T d$. 
	Then $c \le_T d \res \alpha_f <_T d$. 
	Now $c \le_T d \res \alpha_f$ implies that $c \le_L d \res \alpha_f$, 
	and hence $f(c) \le_J f(d \res \alpha_f)$. 
	And $d \res \alpha_f <_T d$ implies that $g(d \res \alpha_f) <_S g(d)$, and hence 
	$g(d \res \alpha_f) <_J g(d)$. 
	By transitivity, $g(c) <_J g(d)$. 
	Secondly, assume that $c$ and $d$ are incomparable. 
	By Lemma \ref{easy lex lemma}, 
	$c <_L d$ implies that $c <_L d \res \h_T(c)$ and hence 
	$f(c) <_J f(d \res \h_T(c))$. 
	On the other hand, $d \res \h_T(c) \le_T d \res \alpha_f$ implies that 
	$d \res \h_T(c) \le_L d \res \alpha_f$, and hence 	
	$f(d \res \h_T(c)) \le_J f(d \res \alpha_f)$. 
	Finally, $d \res \alpha_f <_T d$ implies that $g(d \res \alpha_f) <_S g(d)$, so 
	$g(d \res \alpha_f) <_J g(d)$. 
	By transitivity, $g(c) <_J g(d)$.
	
	\emph{Case c:} $c \in T_{\delta}$ and $d \in T \res \delta$. 
	Now $c \res \alpha_f <_T c$ implies that 
	$c \res \alpha_f <_L c <_L d$. 
	So $f(c \res \alpha_f) <_J f(d)$. 
	In particular, $f(d) \not \le_S f(c \res \alpha_f)$, so $f(d)$ and $f(c \res \alpha_f)$ 
	are incomparable in $S$. 
	But $f(c \res \alpha_f) <_S g(c)$, so by Lemma \ref{easy lex lemma}, 
	$g(c) <_J g(d)$ iff $g(c) \res \alpha_f = f(c \res \alpha_f) <_J f(d)$, which is true.
\end{proof}

\section{Adding Promises, 1}

The general idea of the promise method is that if properness of the forcing fails, 
then there exist a condition and a promise which witnesses this failure. 
On the other hand, this promise can be added to the condition, and the fact that the 
working part of the extended condition fulfills the promise implies that 
there could have been no failure to begin with (this vague argument is crystallized in 
Lemma \ref{extending into dense sets 1} below). 
For this method to work, there must be a way to add promises to conditions, 
a challenge which we handle now.

\begin{defn}
	For any $f \in \Fcal$ and $P \in \Pcal$, we say that $P$ is 
	\emph{suitable for $f$} if there exist $\alpha_f < \delta < \omega_1$, 
	$\langle \delta_n : n < \omega \rangle$, 
	and $\langle \vec x_s : s \in {}^{\underaccent{\breve}{\omega}} \omega \rangle$ satisfying:
	\begin{enumerate}
		\item $P \in \Pcal_\delta$;
		\item $\langle \delta_n : n < \omega \rangle$ is a strictly increasing sequence 
		of ordinals cofinal 
		in $\delta$ with $\delta_0 = \alpha$ and $\delta_n \in E$ for all $n$;
		\item for all $n < \omega$, 
		$\{ \vec x_s : s \in {}^n \omega \}$ is a pairwise disjoint and 
		pairwise compatible subset of 
		$(T_{\delta_n} \otimes S_{\delta_n})^{\otimes m}$, where $m$ 
		is the dimension of $P$;
		\item for all $s \subseteq t$ in ${}^{\underaccent{\breve}{\omega}} \omega$, 
		$\vec x_t \res \delta_{|s|} = \vec x_s$;
		\item for all $s \in {}^{\underaccent{\breve}{\omega}} \omega$, there exists 
		$\vec x_s^+ \in P \cap (T_\delta \otimes S_{\delta})^{\otimes m}$ 
		such that $\vec x_s^+ \res \delta_{|s|} = \vec x_s$;
		\item $f$ is consistent with $\vec x_\emptyset$.
	\end{enumerate}
\end{defn}

\begin{lemma}[Existence of Promises, 1] \label{existence of promises 1}
	Assume the following:
	\begin{enumerate}
		\item $f \in \Fcal$,
		\item $P \in \Pcal$;
		\item there exists $\vec x \in P$ with height greater than or equal to $\alpha_f$ 
		such that $\vec x$ has unique 
		drop-downs to $\alpha_f$ and $f$ is consistent with $\vec x \res \alpha_f$.
	\end{enumerate}
	Then there exists $Q \subseteq P$ in $\Pcal$ such that $Q$ is suitable for $f$.
\end{lemma}

\begin{proof}
	Using Proposition \ref{promise existence start 1}, construct by induction 
	sequences $\langle \delta_n : n < \omega \rangle$ and 
	$\langle \vec x_s : s \in {}^{\underaccent{\breve}{\omega}} \omega \rangle$ satisfying:
	\begin{enumerate}
		\item $\langle \delta_n : n < \omega \rangle$ is a strictly increasing sequence 
		of countable ordinals in $E$ with $\delta_0 = \alpha_f$;
		\item for all $n < \omega$, 
		$\{ \vec x_s : s \in {}^n \omega \}$ is a family of lex-increasing tuples in 
		$T_{\delta_n}^{m}$, where $m$ 
		is the dimension of $P$, which is pairwise disjoint and pairwise compatible;
		\item for all $s \subseteq t$ in ${}^{\underaccent{\breve}{\omega}} \omega$, 
		$\vec x_t \res \delta_{|s|} = \vec x_s$;
		\item $\vec x_{\emptyset} = \vec x \res \alpha_f$.
	\end{enumerate}
	For the base case, let $\delta_0 = \alpha_f$ 
	and $\vec x_{\emptyset} = \vec x \res \alpha_f$. 	
	Now let 
	$n < \omega$ and assume that $\delta_n$ and $\{ \vec x_s : s \in {}^n \omega \}$ 
	are defined as required. 
	Apply Proposition \ref{promise existence start 1} 
	to fix a pairwise disjoint and pairwise compatible family 
	$\{ \vec x_{s,l} : l < \omega \}$ of members of $P$ above $\vec x_s$. 
	Now fix $\delta_{n+1} \in E$ which is greater than $\delta_n$ and greater 
	than the height of 
	$\vec x_{s,l}$ for all $s \in {}^n \omega$ and $l < \omega$. 
	Then for each such $s$ and $l$, choose $\vec x_{s^{\frown} l}$ in $P$ 
	above $\vec x_{s,l}$ with height $\delta_{n+1}$. 
	This completes the induction. 
	Let $\delta = \sup_n \delta_n$. 
	For each $s \in {}^{\underaccent{\breve}{\omega}} \omega$, 
	pick some $\vec x_s^+ \in P$ with height $\delta$ 
	above $\vec x_s$. 
	Now define 
	$Q = \bigcup \{ P_{\vec x_s^+} : s \in {}^{\underaccent{\breve}{\omega}} \omega \}$.
\end{proof}

\begin{prop}[Adding Promises, 1] \label{adding promises 1}
	Assume the following:
	\begin{enumerate}
		\item $f \in \Fcal$;
		\item $\Gamma \subseteq \Pcal$ is countable and $f$ fulfills 
		every member of $\Gamma$;
		\item $P \in \Pcal$ is suitable for $f$ with witnesses 
		$\delta$, $\langle \delta_n : n < \omega \rangle$ 
		and $\langle \vec x_s : s \in {}^{\underaccent{\breve}{\omega}} \omega \rangle$;
		\item $f$ is consistent with $\vec x_\emptyset$.
	\end{enumerate}
	Then there exists $g \in \Fcal_{\delta_U}$ such that 
	$f \subseteq g$ and $g$ fulfills every member of $\Gamma \cup \{ P \}$.
\end{prop}

\begin{proof}
	Let $m$ be the dimension of $P$. 
	Let $\langle P_n : n < \omega \rangle$ be an enumeration of 
	$\Gamma$ in which each member appears infinitely many times, and 
	let $\langle c_n : n < \omega \rangle$ be an enumeration of $T_\delta$. 
	We define by induction $\alpha_n$, $\vec z_n$, $s_n$, and $f_n$ for 
	$n < \omega$, maintaining the following inductive hypotheses:
	\begin{itemize}
		\item $\delta_n \le \alpha_n < \omega_1 \cap \delta$;
		\item $\vec z_n$ is a lex-increasing tuple in 
		$(T_\delta \times T_{\delta})^{<\omega}$ with 
		unique drop-downs to $\alpha_n$;
		\item $s_n \in {}^{\underaccent{\breve}{\omega}} \omega$, $\vec x_{s_n}$ has height $\alpha_n$, 
		and for $n > 0$, $\vec x_{s_n}$ and $\vec z_n \res \alpha_n$ are disjoint 
		and compatible;
		\item $k < n$ implies $s_k \subseteq s_n$;
		\item $f_n \in \Fcal$;
		\item $\alpha_{f_n} = \alpha_n$;
		\item $f_n$ is consistent with $\vec z_n \res \alpha_n$ and $\vec x_{s_n}$;
		\item $f_n$ fulfills every member of $\Gamma$.
	\end{itemize}

	For the base case, let $\alpha_0 = \alpha_f$, 
	$\vec z_0$ is the empty tuple, $s_0 = \emptyset$, 
	and $f_0 = f$. 
	Now assume that $n < \omega$ and for all $m \le n$,  
	$\alpha_m$, $\vec z_m$, $s_m$, and $f_m$ are defined as required. 
	Write 
	$\vec z_n = ((a^n_k,b^n_k) : k < l)$. 
	Let $\alpha_{n+1}$ be the least ordinal satisfying:
	\begin{itemize}
		\item $\alpha_{n+1} > \alpha_n$;
		\item $\alpha_{n+1}$ is equal to $\delta_q$ for some $q \ge n+1$;
		\item $\{ a^n_k : k < l \} \cup \{ c_n \}$ has unique drop-downs to $\alpha_{n+1}$.
	\end{itemize}
	Pick $s_{n,0}$ and $s_{n,1}$ in ${}^q \omega$ such that for each $i < 2$, 
	$s_n \subseteq s_{n,i}$, $\vec x_{s_{n,i}}$ is disjoint from 
	$\vec z_n \res \alpha_{n+1}$, and $c_n \res \alpha_{n+1}$ is not the first coordinate 
	of any pair in $\vec x_{s_{n,i}}$.

	Apply Proposition \ref{extension 1} to fix $f_{n+1} \in \Fcal$ 
	such that $\alpha_{f_{n+1}} = \alpha_{n+1}$, 
	$f_n \subseteq f_{n+1}$, 
	$f_{n+1}$ fulfills every member of $\Gamma$, 
	and $f_{n+1}$ is consistent with $\vec z_n \res \alpha_{n+1}$, 
	$\vec x_{s_{n,0}}$, and $\vec x_{s_{n,1}}$. 
	Define $s_{n+1} = s_{n,1}$. 
	Pick some $\vec x_{s_{n,0}}^+ \in P$ above $\vec x_{s_{n,0}}$ with height $\delta$.
	
	We define $\vec z_{n+1}$ in three steps. 	
	For the first step, apply Lemma \ref{compatible drop 1 bonus} and 
	let $\vec z_n'$ be the lex-increasing tuple which lists the elements 
	of $\vec z_n$ and $\vec x_{s_{n,0}}^+$. 
	For the second step, we first ask whether $c_n$ is equal to $a_k^n$ for some $k < l$. 
	If so, then $c_n$ has already been handled, and we let $\vec z_n'' = \vec z_n'$ and 
	move on to step three. 
	Suppose not. 
	Since $f_{n+1}$ is consistent with $\vec z_n \res \alpha_{n+1}$ and is injective, 
	$f_{n+1}(c_n \res \alpha_{n+1})$ is not equal to $b_k^n \res \alpha_{n+1}$ for any $k < l$. 
	Similarly, $f_{n+1}(c_n \res \alpha_{n+1})$ is not the second coordinate of 
	any pair in $\vec x_{s_{n,i}}$, 
	for each $i < 2$. 
	Pick some $d \in S_{\delta}$ which is above $f_{n+1}(c_n \res \alpha_{n+1})$. 
	Applying Lemma \ref{compatible drop 1 bonus}, 
	let $\vec z_n''$ be the lex-increasing tuple which lists the elements of 
	$\vec z_n'$ together with $(c_n,d)$. 

	Now we describe step three. 
	Since $f_{n+1}$ fulfills $P_n$, fix some 
	$\vec w \in P_n$ of height $\alpha_{n+1}$ 
	which is consistent with $f_{n+1}$ and disjoint from $\vec z_n'' \res \alpha_{n+1}$ 
	and $\vec x_{s_{n,1}}$. 
	Fix $\vec w^+ \in P_n$ above $\vec w$ of height $\delta$. 
	Applying Lemma \ref{compatible drop 1 bonus}, 
	let $\vec z_{n+1}$ be the lex-increasing tuple which lists the elements of 
	$\vec w^+$ and $\vec z_n''$.

	This completes the induction. 
	Define $g \in \Fcal$ by letting $g \res (T \res \delta) = \bigcup_n f_n$ 
	and $g(c) = d$ iff for some $n < \omega$, 
	the pair $(c,d)$ appears in $\vec z_n$. 
	By Lemma \ref{assisting lemma}, $g \in \Fcal$, and $g$ fulfills $P$ as witnessed 
	by $\{ \vec x_{s_{n,0}}^+ : n < \omega \}$.
\end{proof}

\section{The First Poset}

We are now ready to introduce and derive the main properties of 
the forcing which we use to prove Theorem 1.

\begin{defn}
	Define $\p(T,S)$ to be the forcing consisting of conditions which are 
	ordered pairs $p = (f_p,\Gamma_p)$ satisfying:
	\begin{enumerate}
		\item $f_p \in \Fcal$;
		\item $\Gamma_p \subseteq \Pcal$ is countable;
		\item $f_p$ fulfills every member of $\Gamma_p$.
	\end{enumerate}
	For any $p \in \p(T,S)$, let $c_p = c_{f_p}$ and $\alpha_p = \alpha_{f_p}$. 
	Define $q \le p$ if 
	$f_p \subseteq f_q$ and $\Gamma_p \subseteq \Gamma_q$.
\end{defn}

The next two lemmas follow immediately from 
Lemmas \ref{extension 1} and \ref{adding promises 1} respectively.

\begin{lemma}[Poset Extension, 1] \label{extension poset 1}
	Assume the following:
	\begin{enumerate}
		\item $p \in \p(T,S)$;
		\item $\alpha_p < \delta \in E$;
		\item $\vec x$ is a lex-increasing tuple in 
		$(T_\delta \otimes S_{\delta})^{<\omega}$ with unique drop-downs to $\alpha_p$;
		\item $f_p$ is consistent with $\vec x \res \alpha_p$.
	\end{enumerate}
	Then there exists $q \le p$ such that $\alpha_q = \delta$ and $f_q$ 
	is consistent with $\vec x$.
\end{lemma}

It follows that, assuming $\p(T,S)$ preserves $\omega_1$, 
any generic filter on $\p(T,S)$ yields a function as described in 
Section \ref{Compatibility of Countryman Lines}.

\begin{lemma}[Poset Adding Promises, 1] \label{adding promises to conditions 1}
	Assume the following:
	\begin{enumerate}
		\item $p \in \p(T,S)$;
		\item $P \in \Pcal$;
		\item $P$ is suitable for $f_p$.
	\end{enumerate}
	Then there exists $q \le p$ such that $P \in \Gamma_q$.
\end{lemma}

\begin{lemma} \label{extending into dense sets 1}
	Assume that $L^*$ and $J$ are not near. 
	Let $\lambda \ge (2^{|\p(T,S)|})^+$ be a regular cardinal and let 
	$N \prec H(\lambda)$ be countable such that $T$, $S$, $E$, and $\p(T,S)$ are in $N$. 
	Let $\delta = N \cap \omega_1$. 
	Assume the following:
	\begin{enumerate}
		\item $p \in N \cap \p(T,S)$;
		\item $D \in N$ is a dense subset of $\p(T,S)$;
		\item $\vec x$ is a lex-increasing tuple in 
		$(T_\delta \otimes S_\delta)^{<\omega}$ with unique drop-downs to $\alpha_p$;
		\item $f_p$ is consistent with $\vec x \res \alpha_p$.
	\end{enumerate}
	Then there exists $q \le p$ in $N \cap D$ such that 
	$f_q$ is consistent with $\vec x \res \alpha_q$.
\end{lemma}

\begin{proof}
	Suppose for a contradiction that the conclusion fails. 
	Let $n$ be the dimension of $\vec x$. 
	Define $P_0$ as the set of all lex-increasing tuples $\vec y$ in 
	$(((T \res E) \otimes (S \res E))^{\otimes n})_{\vec x \res \alpha_p}$ 
	such that for any $q \le p$ in $D$ with $\alpha_q$ less than 
	or equal to the height of $\vec y$, 
	$f_q$ is not consistent with $\vec y \res \alpha_q$. 
	Note that $P_0 \in N$ by elementarity, and $P_0$ is a downwards closed subset of the tree 
	$(((T \res E) \otimes (S \res E))^{\otimes n})_{\vec x \res \alpha_p}$. 
	We claim that for all $\alpha_p \le \gamma < \delta$, $\vec x \res \gamma \in P_0$. 
	Otherwise, there exists $q \le p$ in $D$ such that $\alpha_q \le \gamma$ 
	and $f_q$ is consistent with $(\vec x \res \gamma) \res \alpha_q = 
	\vec x \res \alpha_q$. 
	Since $p$, $D$, $\gamma$, and $\vec x \res \gamma$ are in $N$, by elementarity 
	we may assume that $q \in N$. 
	But then the conclusion of the lemma holds, which is a contradiction. 
	It follows that $P_0$ is uncountable. 
	Let $P \in N$ be an uncountable downwards closed subset of $P_0$ such that 
	every element of $P$ has uncountably many members of $P$ above it. 
	Note that $P \in \Pcal$ and $\vec x \res \alpha_p \in P$.	
	
	By Lemma \ref{existence of promises 1}, we can fix $Q \in \Pcal$ 
	such that $Q \subseteq P$ and $Q$ is suitable for $f_p$. 
	Since the root of $P_0$ is $\vec x \res \alpha_p$, 
	$\vec x \res \alpha_p$ is the only element of $P$ of height $\alpha_p$. 
	Hence, $\vec x \res \alpha_p$ is equal to $\vec x_\emptyset$, where 
	$\{ \vec x_s : s \in {}^{\underaccent{\breve}{\omega}} \omega \}$ 
	witnesses that $Q$ is suitable for $f_p$. 
	Applying Lemma \ref{adding promises to conditions 1}, 
	fix $q \le p$ such that $Q \in \Gamma_q$. 
	Since $D$ is dense, we can fix $r \le q$ in $D$. 
	As $f_r$ fulfills $Q$, there exists 
	$\vec y \in Q \cap (T_{\alpha_r} \otimes S_{\alpha_r})^{\otimes n}$ 
	such that $f_r$ is consistent with $\vec y$. 
	Since $Q \subseteq P_0$, $\vec y \in P_0$. 
	But $r \le p$ is in $D$ and 
	$f_r$ is consistent with $\vec y$, which contradicts the definition of $P_0$.
\end{proof}

\begin{thm} \label{totally proper 1}
	The forcing $\p(T,S)$ is totally proper.
\end{thm}

\begin{proof}
	Let $\lambda \ge (2^{|\p(T,S)|})^+$ be a regular cardinal and let 
	$N \prec H(\lambda)$ be countable such that $T$, $S$, $E$, and 
	$\p(T,S)$ are in $N$. 
	Let $\delta = N \cap \omega_1$. 
	Consider $p \in N \cap \p(T,S)$. 
	We claim that there exists $q \le p$ such that $q$ is a total master condition 
	for $N$ and $\p(T,S)$. 
	Specifically, we build by induction a descending sequence of conditions 
	$\langle p_n : n < \omega \rangle$ in $N \cap \p(T,S)$ such that for every dense 
	set $D \in N$, there exists some $n$ with $p_n \in D$, together with 
	a lower bound $q$ of this sequence.
	
	Fix the following objects:
	\begin{itemize}
		\item an enumeration $\langle P_n : n < \omega \rangle$ of 
		$N \cap \Pcal$ in which each element appears infinitely many times;
		\item an enumeration $\langle c_n : n < \omega \rangle$ of $T_\delta$;
		\item a sequence $\langle D_n : n < \omega \rangle$ which lists every 
		dense open subset 
		of $\p(T,S)$ which lies in $N$.
	\end{itemize}
	We define by induction a descending sequence $\langle p_n : n < \omega \rangle$ 
	of conditions in $N \cap \p(T,S)$ together with a sequence 
	$\langle \vec x_n : n < \omega \rangle$ of lex-increasing tuples in $T_\delta^{<\omega}$ 
	as follows.
	
	For the base case, let $p_0 = p$ and let $\vec x_0$ be the empty tuple. 
	Now assume that $p_n$ and $\vec x_n$ are defined for some $n < \omega$ 
	such that $\vec x_n$ has 
	unique drop-downs to $\alpha_{p_n}$ 
	and $f_{p_n}$ is consistent with $\vec x_n \res \alpha_{p_n}$. 
	Write $\vec x_n = ((a^n_k,b^n_k) : k < l)$.
	We define $p_{n+1}$ and $\vec x_{n+1}$ in three steps. 	
	In the first step, apply Lemma \ref{extending into dense sets 1} to fix 
	$p_{n+1} \le p_n$ in $N \cap D_n$ such that 
	$f_{p_{n+1}}$ is consistent with $\vec x_n \res \alpha_{p_{n+1}}$. 
	Moreover, using Lemma \ref{extension poset 1} we may assume without loss of 
	generality that $\alpha_{p_{n+1}}$ is high enough that 
	the set $\{ a^n_k : k < l \} \cup \{ c_n \}$ has unique 
	drop-downs to $\alpha_{p_{n+1}}$.
	
	For the second step, 
	we ask whether $c_n$ is equal to $a_k^n$ for some $k < l$. 
	If yes, then $c_n$ has already been handled, so we let $\vec x_n' = \vec x_n$ and 
	move on to step three. 
	Suppose not. 
	Since $f_{p_{n+1}}$ is consistent with 
	$\vec x_n \res \alpha_{p_{n+1}}$ and is injective, 
	$f_{p_{n+1}}(c_n \res \alpha_{p_{n+1}})$ is not equal to 
	$b_k^n \res \alpha_{p_{n+1}}$ for any $k < l$. 
	Pick some $d \in S_{\delta}$ which is above 
	$f_{p_{n+1}}(c_n \res \alpha_{p_{n+1}})$. 	
	By Lemma \ref{easy lex lemma} and the fact that $f_{n+1}$ is consistent with 
	$\vec x_n \res \alpha_{n+1}$, for all $k < l$ we have that 
	\begin{multline*}
	a_k^n <_L c_n \iff a_k^n \res \alpha_{p_{n+1}} <_L c_n \res \alpha_{p_{n+1}} \\ \iff 
	b_k^n \res \alpha_{p_{n+1}} <_J f_{n+1}(c_n \res \alpha_{p_{n+1}}) \iff 
	b_k^n <_J d.
	\end{multline*}
	Hence, the tuples $\vec x_n$ and $((c_n,d))$ are compatible. 
	Let $\vec x_n'$ be the lex-increasing tuple which lists the elements 
	of $\vec x_n$ together with $(c_n,d)$. 

	Now we describe step three. 
	If $P_n \notin \Gamma_{p_{n+1}}$, then let $\vec x_{n+1} = \vec x_n'$ and we are done. 
	Otherwise, since $f_{p_{n+1}}$ fulfills $P_n$, fix some $\vec w \in P_n$ 
	of height $\alpha_{p_{n+1}}$ which is disjoint from 
	$\vec x_n' \res \alpha_{p_{n+1}}$ 
	such that $f_{p_{n+1}}$ is consistent with $\vec w$. 
	Fix $\vec w^+ \in P_n$ above $\vec w$ of height $\delta$. 
	By Lemma \ref{compatible drop 1 bonus}, 
	we can let $\vec x_{n+1}$ be the lex-increasing tuple which lists the elements of 
	$\vec w^+$ and $\vec x_n'$.

	This completes the induction. 
	Note that by Lemma \ref{extension poset 1} and 
	the fact that the decreasing sequence meets every dense set in $N$, 
	$\sup_{n} \alpha_{p_n} = \delta$. 
	Define $q$ as follows. 
	Let $f_{q} \res (T \res \delta) = \bigcup_n f_{p_n}$ and 
	for all $c \in T_\delta$, $f_q(c) = d$ iff for some $n < \omega$, 
	$(c,d)$ appears in $\vec x_n$. 
	By Lemma \ref{assisting lemma}, $f_q \in \Fcal$. 
	Let $\Gamma_q = \bigcup_n \Gamma_{p_n}$. 
	Then $q = (f_q,\Gamma_q)$ is in $\p(T,S)$ and is a lower bound of 
	$\langle p_n : n < \omega \rangle$.
\end{proof}

\section{Making $\Ucal(T)$ an Ultrafilter} \label{Making U(T) an Ultrafilter}

We now turn to the second main theorem of the article, which is the consistency 
of the statement that \textsf{CH} holds and 
for any coherent Aronszajn tree $T \subseteq {}^{\underaccent{\breve}{\omega}_1} \omega$, 
$\Ucal(T)$ is an ultrafilter. 
Fix such a tree $T$ and consider a set $X \subseteq \omega_1$. 
If $\omega_1 \setminus X \in \Ucal(T)$, then $X$ is already handled. 
Otherwise, every member of $\Ucal(T)$ has non-empty intersection with $X$, 
that is, $X \in \Ucal(T)^+$. 
In this case, our goal is to define a forcing $\p(T,X)$ which is totally proper and 
adds an uncountable antichain $A \subseteq T$ such that $\Delta (A) \subseteq X$. 
Rather than adding the antichain explicitly, we design a forcing which adds 
an uncountable subtree which splits only at levels in $X$. 
This approach is justified by the following lemma.

\begin{lemma}
	Suppose that $T \subseteq {}^{\underaccent{\breve}{\omega}_1} \omega$ is a non-Suslin 
	coherent Aronszajn tree. 
	For any set $X \subseteq \omega_1$, $X \in \Ucal(T)$ iff there exists an uncountable 
	downwards closed subtree $U$ of $T$ such that for every $x \in U$, if $x$ 
	splits in $U$ then $\h_T(x) \in X$.
\end{lemma}

\begin{proof}
	Suppose that $X \in \Ucal(T)$. 
	Fix an uncountable antichain $A \subseteq T$ 
	such that $\Delta(A) \subseteq X$. 
	Let $U = \{ x \in T : \exists y \in A \; x <_T y \}$. 
	Then $U$ is uncountable. 
	For any splitting element $x$ of $U$, the height of $x$ is equal to $\Delta(y,z)$ for 
	any $y, z \in A$ which are above two distinct immediate successors of $x$ in $U$. 
	Hence, the height of any splitting element of $U$ is in $X$. 
	Conversely, suppose that $U$ is an uncountable downwards closed subtree of $T$ 
	such that the height of any splitting element of $U$ is in $X$. 
	Since $T$ is non-Suslin, we can fix an uncountable antichain $A \subseteq U$. 
	If $y$ and $z$ are distinct members of $A$, then 
	$y \res \Delta(y,z) = z \res \Delta(y,z)$ splits in $U$, 
	and hence its height, which is $\Delta(y,z)$, is in $X$. 
\end{proof}

We now describe how to obtain the desired consistency result. 
We iterate with countable support of length $\omega_2$ over a model of 
$2^\omega = \omega_1$ and $2^{\omega_1} = \omega_2$, 
alternating between forcings which specialize Aronszajn trees and 
forcings of the form $\p(T,X)$, where $T$ is as above and $X \in \Ucal(T)^+$. 
Both types of forcings have size 
$2^{\omega_1}$ and are $(<\! \omega_1)$-proper, Dee-complete, and $\omega_2$-p.i.c. 
Bookkeeping to handle all relevant trees and sets, we arrange that the final 
iteration forces that all Aronszajn trees are special and that 
$\Ucal(T)$ is an ultrafilter for any coherent Aronzajn tree 
$T \subseteq {}^{\underaccent{\breve}{\omega}_1} \omega$. 
As in the first application, the core result about $\p(T,X)$ is that it is totally proper, 
and we leave the verification of the other properties to the interested reader 
(see the comments towards the end of Section \ref{Compatibility of Countryman Lines}). 
The structure of our proof is quite similar to that of Theorem 1.

We make use of the following basic result about coherent Aronszajn trees. 
For a proof, see the arguments of \cite[Lems. 4.2.4, 4.2.5]{todorbook}.

\begin{lemma}[Fundamental Lemma on Coherent Trees] \label{coherent trees lemma}
	Suppose that $T \subseteq {}^{\underaccent{\breve}{\omega}_1} \omega$ is a coherent Aronszajn tree 
	which is non-Suslin. 
	Let $\langle \vec x_\alpha : \alpha < \omega_1 \rangle$ be a sequence of distinct 
	elements of $T^{n}$ for some $0 < n < \omega$. 
	Write each $\vec x_\alpha$ as 
	$( a_0^\alpha, \ldots, a_{n-1}^\alpha )$. 
	Then there exists an uncountable set $Y \subseteq \omega_1$ such that 
	for all $\alpha < \beta$ in $Y$, 
	there exists $\xi < \omega_1$ such that for all $k < n$, 
	$a_k^\alpha$ and $a_k^\beta$ are incomparable and 
	$\Delta(a^\alpha_k,a^\beta_k) = \xi$.
\end{lemma}

For the remainder of this application, fix a coherent Aronszajn 
tree $T \subseteq {}^{\underaccent{\breve}{\omega}_1} \omega$ which is non-Suslin, and fix 
a set $X \subseteq \omega_1$.

\section{Tuples, 2} \label{Tuples, 2}

We work with injective tuples of the form 
$\vec x = (a_0,\ldots,a_{n-1})$, where 
for some $\alpha < \omega_1$, $a_0,\ldots,a_{n-1}$ are in $T_\alpha$. 
We refer to $\alpha$ as the \emph{height of $\vec x$}. 
For $\xi \le \alpha$, let $\vec x \res \xi = (a_0 \res \xi,\ldots,a_{n-1} \res \xi)$. 
We say that $\vec x$ has \emph{unique drop-downs to $\xi$} if 
$\vec x \res \xi$ is injective.

\begin{defn}
	Let $\vec x = (a_0,\ldots,a_{n-1})$ be an injective tuple 
	in $T_\alpha^n$, for some $\alpha < \omega_1$ and $0 < n < \omega$. 
	We say that \emph{$\vec x$ has meets in $X$} if for all 
	$k < l < n$, $\Delta(a_k,a_l) \in X$.  
\end{defn}

\begin{lemma} \label{meets in X drop}
	Let $\xi < \alpha < \omega_1$ and suppose that $\vec x$ is an injective tuple 
	in $T_\alpha^{<\omega}$ which has 
	unique drop-downs to $\xi$. 
	Then $\vec x$ has meets in $X$ iff $\vec x \res \xi$ has meets in $X$.
\end{lemma}

\begin{defn}
	Let $\vec x = (a_0,\ldots,a_{n-1})$ and $\vec y = (b_0,\ldots,b_{m-1})$ be 
	injective tuples in $T_\alpha^{<\omega}$ 
	for some $\alpha < \omega_1$. 
	\begin{enumerate}
		\item The tuples $\vec x$ and $\vec y$ are \emph{disjoint} 
		if $a_i \ne b_j$ for all 
		$i < n$ and $j < m$;
		\item If $\vec x$ and $\vec y$ are disjoint, then we say that they are 
		\emph{compatible} if for all $i < n$ and $j < m$, $\Delta(a_i,b_j) \in X$.
	\end{enumerate}
\end{defn}
 
Note that $\vec x$ and $\vec y$ are compatible iff there exists an injective 
tuple in $T^{<\omega}_\alpha$ with meets in $X$ 
which lists the elements of $\vec x$ and $\vec y$.

\begin{lemma} \label{compatible drop 2}
	Let $\xi < \alpha < \omega_1$. 
	Suppose that $\vec x$ and $\vec y$ are disjoint injective tuples in 
	$T_\alpha^{<\omega}$, 
	each with unique drop-downs to $\xi$, 
	such that $\vec x \res \xi$ and $\vec y \res \xi$ are disjoint. 
	Then $\vec x$ and $\vec y$ are compatible iff $\vec x \res \xi$ and 
	$\vec y \res \xi$ are compatible.
\end{lemma}

Suppose that 
$\vec x = (a_0,\ldots,a_{n-1})$ and $\vec y = (b_0,\ldots,b_{n-1})$ are in $T^{\otimes n}$, 
not necessarily of the same height. 
We say that $\vec x$ and $\vec y$ are \emph{componentwise incomparable} if 
for all $k < n$, $a_k$ and $b_k$ are incomparable in $T$. 
Componentwise incomparable tuples $\vec x$ and $\vec y$ are said to have 
\emph{uniform meets} if there exists some $\gamma < \omega_1$, which we denote 
by $\Delta(\vec x,\vec y)$, such that for all $k < n$, 
$\Delta(a_k,b_k) = \gamma$. 
The tuples $\vec x$ and $\vec y$ are \emph{fully incomparable} if for all $k, m < n$, 
$a_k$ and $b_m$ are incomparable in $T$.

\section{Promises, 2} \label{Promises, 2}

The type of promises which we use for the proof of Theorem 2 is described next.

\begin{defn}[Promises, 2]
	Define $\Pcal$ to be the set of all $P$ satisfying that for some 
	$\gamma < \omega_1$ and for some $0 < n < \omega$:
	\begin{enumerate}
		\item $P$ is a non-empty downwards closed subset of the tree 
		$[T \res (\omega_1 \setminus \gamma)]^{\otimes n}$;
		\item every member of $P$ has uncountably many elements of $P$ above it;
		\item every member of $P$ has meets in $X$.
	\end{enumerate}
	In the above, $\gamma$ is called the \emph{base level of $P$} and 
	$n$ is called the \emph{dimension of $P$}. 
	The set of $P \in \Pcal$ such that $P$ has base level $\gamma$ is denoted by $\Pcal_\gamma$.
\end{defn}

\begin{prop} \label{promise existence start 2}
	For all $P \in \Pcal$ and for any $\vec x \in P$, there exists a family 
	$\{ \vec x_n : n < \omega \}$ of elements of $P$ above $\vec x$ with the same height 
	which is pairwise disjoint, pairwise compatible, and any two members of the 
	family have uniform meets.
\end{prop}

\begin{proof}
	Write $\vec x = (a_0,\dots,a_{n-1})$. 
	For each $\alpha < \omega_1$ greater than the height of $\vec x$, 
	pick some $\vec x_\alpha \in P$ above $\vec x$ with height $\alpha$. 
	Write each $\vec x_\alpha$ as $(a_0^\alpha,\ldots,a_{n-1}^\alpha)$. 
	Since $\vec x$ is injective and $a_k <_T a_k^\alpha$ 
	for all $\alpha$ and for all $k < n$, 
	it follows that for all distinct $k, l < n$ and for all $\alpha \le \beta$, 
	$a_k^\alpha$ and $a_l^\beta$ are incomparable in $T$. 
	Applying Lemma \ref{coherent trees lemma}, 
	fix an uncountable set $Y \subseteq \omega_1$ 
	so that for all $\alpha < \beta$ in $Y$, $\vec x_\alpha$ and $\vec x_\beta$ 
	are componentwise incomparable and 
	$\vec x_\alpha$ and $\vec x_\beta$ have uniform meets. 
	Then for all $\alpha < \beta$ in $Y$, $\vec x_\alpha$ and $\vec x_\beta$ 
	are fully incomparable, and hence $\vec x_\alpha$ and $\vec x_\beta \res \alpha$ 
	are disjoint and have uniform meets.

	Define a function $F : [Y]^2 \to 2$ so that for all $\alpha < \beta$ in $Y$, 
	$F(\{ \alpha, \beta \})$ is equal to $1$ if 
	$\vec x_\alpha$ and $\vec x_\beta \res \alpha$ are compatible, and $0$ otherwise. 
	By the Dushnik-Miller theorem, there exists a set $Z \subseteq Y$ such that 
	either (a) $Z$ is infinite and for all $\alpha < \beta$ in $Z$, 
	$F(\{ \alpha, \beta \}) = 1$, 
	or else (b) $Z$ is uncountable and for all $\alpha < \beta$ in $Z$, 
	$F(\{ \alpha,\beta \}) = 0$. 

	Suppose for a contradiction that (b) holds. 
	Consider any $\alpha < \beta$ in $Z$. 
	Since $F(\{ \alpha,\beta \}) = 0$, $\vec x_\alpha$ and $\vec x_\beta \res \alpha$ 
	are not compatible. 
	But for all distinct $k, m < n$, 
	$\Delta(a_k^\alpha,a_m^\beta) = \Delta(a_k,a_m) \in X$. 
	So there must exist some $k < n$ such that $\Delta(a^\alpha_k,a^\beta_k) \notin X$. 
	Since $\vec x_\alpha$ and $\vec x_\beta$ have uniform meets, 
	$\Delta(a^\alpha_0,a^\beta_0) = \Delta(a^\alpha_k,a^\beta_k) \notin X$. 
	To summarize, for all $\alpha < \beta$ in $Z$, 
	$\Delta(a^\alpha_0,a^\beta_0) \notin X$. 
	So $A = \{ a^\alpha_0 : \alpha \in Z \}$ is an uncountable antichain 
	of $T$ such that $\Delta (A) \subseteq \omega_1 \setminus X$. 
	By the definition of $\Ucal(T)$, it follows that $\omega_1 \setminus X \in \Ucal(T)$, 
	which contradicts that $X \in \Ucal(T)^+$.

	So (a) holds. 
	Pick some $\delta \in E$ greater than the height of 
	$\vec x_\alpha$ for all $\alpha \in Z$. 
	For each $\alpha \in Z$, choose some $\vec x_\alpha^+ \in P$ 
	with height $\delta$ which is above $\vec x_\alpha$. 
	Now for all $\alpha < \beta$ in $Z$, $\vec x_\alpha$ and 
	$\vec x_\beta \res \alpha$ are disjoint, compatible, and have uniform meets. 
	By Lemma \ref{compatible drop 2}, $\vec x_\alpha^+$ and $\vec x_\beta^+$ 
	are disjoint, compatible, and have uniform meets.
\end{proof}

\section{The Working Part, 2} \label{The Working Part, 2}

The forcing $\p(T,X)$ adds the desired subtree by countable approximations. 
The next definition describes these objects.

\begin{defn}
	Define $\Fcal$ to be the set of functions of the form 
	$f : T \res (\alpha+1) \to 2$, for some $\alpha < \omega_1$, satisfying:
	\begin{enumerate}
		\item $f(\emptyset) = 1$;
		\item if $f(x) = 1$, then there exists $y \in T_\alpha$ such that 
		$x \le_T y$ and $f(x) = 1$;
		\item for all $x, y \in T \res (\alpha+1)$:
		\begin{enumerate}
			\item if $f(y) = 1$ and $x <_T y$, then $f(x) = 1$;
			\item if $x$ and $y$ are incomparable in $T$ and $f(x) = f(y) = 1$, 
			then $\Delta(x,y) \in X$.
		\end{enumerate}
	\end{enumerate}
	For any $\alpha < \omega_1$, 
	define $\Fcal_\alpha$ to be the set of $f \in \Fcal$ such that 
	$\dom(f) = T \res (\alpha+1)$.
\end{defn}

\begin{defn}
	Suppose that $\xi \le \alpha < \omega_1$, 
	$f \in \Fcal_\alpha$, and 
	$\vec x = (a_0,\ldots,a_{n-1})$ is an injective tuple in $T_\xi^n$. 
	We say that \emph{$f$ is consistent with $\vec x$} if for all $i < n$, $f(a_i) = 1$.
\end{defn}

\begin{defn}
	Suppose that $\gamma \le \alpha < \omega_1$, 
	$f \in \Fcal_\alpha$, and $P \in \Pcal_\gamma$. 
	We say that $f$ \emph{fulfills $P$} if there exists a pairwise disjoint 
	family $\{ \vec x_n : n < \omega \}$ of members of $P$ 
	with height $\alpha$ such that for all $n$, 
	$f$ is consistent with $\vec x_n$.
\end{defn}

\section{Extension, 2} \label{Extension, 2}

The next three lemmas allows us to stretch a condition up arbitrarily high in the tree. 
Lemma \ref{successor case 2} handles the successor case and 
Lemma \ref{extension 2} handles the general case. 
We also need Lemma \ref{splitting extension 2} in order to add promises to a condition 
(Proposition \ref{adding promises 2}).

\begin{lemma} \label{successor case 2}
	Assume the following:
	\begin{enumerate}
		\item $\alpha < \omega_1$ and $f \in \Fcal_\alpha$;
		\item $\Gamma \subseteq \Pcal$ is countable and $f$ fulfills every member of $\Gamma$;
		\item $\vec x$, $\vec y$, and $\vec z$ are injective tuples in 
		$T_{\alpha+1}^{<\omega}$, each with unique drop-downs to $\alpha$;
		\item either $\vec y = \vec z = \emptyset$, 
		or $\vec y$ and $\vec z$ are non-empty, disjoint, and compatible;
		\item $\vec y \res \alpha = \vec z \res \alpha$;
		\item $\vec x \res \alpha$ and $\vec y \res \alpha$ are disjoint;
		\item $f$ is consistent with $\vec x \res \alpha$ and $\vec y \res \alpha$.
	\end{enumerate}
	Then there exist $g \in \Fcal_{\alpha+1}$ such that:
	\begin{enumerate}
		\item[(a)] $f \subseteq g$;
		\item[(b)] $g$ fulfills every member of $\Gamma$;
		\item[(c)] $g$ is consistent with $\vec x$, $\vec y$, and $\vec z$.
	\end{enumerate}
\end{lemma}

\begin{proof}
	Write $\vec y = (a^0_k : k < n)$ and $\vec z = (a^1_k : k < n)$. 
	Let $\langle P_m : m < \omega \rangle$ be an enumeration of $\Gamma$ 
	so that every member appears infinitely often. 
	Since $f$ fulfills every member of $\Gamma$, we can 
	fix a pairwise disjoint sequence 
	$\langle \vec w_m : m < \omega \rangle$ satisfying that for all $m < \omega$:
	\begin{enumerate}
		\item $\vec w_m$ is in $P_m$ and has height $\alpha$;
		\item $\vec w_m$ is disjoint from $\vec x \res \alpha$ and 
		$\vec y \res \alpha$;
		\item $f$ is consistent with $\vec w_m$.
	\end{enumerate}
	For each $m < \omega$, fix $\vec w_m^+$ in $P_m$ of height $\alpha+1$ 
	which is above $\vec w_m$. 
	It suffices to define some $g \in \Fcal_{\alpha+1}$ 
	such that $f \subseteq g$ and $g$ is consistent $\vec x$, 
	$\vec y$, $\vec z$, and $\vec w_m^+$ for all $m < \omega$.
	
	To construct $g$, we consider an arbitrary $x \in T_\alpha$ and define 
	$g$ on the immediate successors of $x$. 
	If $f(x) = 0$, then define $g(y) = 0$ for every immediate successor $y$ of $x$. 
	Suppose that $f(x) = 1$. 
	We split into four mutually exclusive cases.

	\emph{Case 1:} There exists $m < \omega$ such that $x$ appears in $\vec w_m$. 
	Write $\vec w_m = (a_0,\ldots,a_{l-1})$ and $\vec w_m^+ = (a_0^+,\ldots,a_{l-1}^+)$, and 
	fix $k < l$ such that $a_k = x$. 
	Define $g(a_k^+) = 1$, and $g(y) = 0$ 
	for any other immediate successor $y$ of $x$.
	
	\emph{Case 2:} $x$ appears in $\vec x \res \alpha$. 
	Proceed similarly as in case 1.
	
	\emph{Case 3:} $x$ appears in $\vec y \res \alpha$. 
	In this case, $\vec y$ and $\vec z$ are non-empty, disjoint, and compatible. 
	In particular, since $\vec y \res \alpha = \vec z \res \alpha$, 
	$\Delta(a^0_0,a^1_0) = \alpha \in X$. 
	Fix $k < n$ such that $x = a_k^0 \res \alpha$. 
	Then also $x = a_k^1 \res \alpha$. 
	Define $g(a_k^0) = 1$, $g(a_k^1) = 1$, and $g(y) = 0$ for any other 
	immediate successor $y$ of $x$. 

	\emph{Case 4:} None of the above. In this case, arbitrarily choose some immediate 
	successor $y$ of $x$, and define $g(z) = 1$ iff $z = y$ for any immediate 
	successor $z$ of $x$.
\end{proof}

\begin{prop}[Extension, 2] \label{extension 2}
	Assume the following:
	\begin{enumerate}
		\item $\alpha \le \delta < \omega_1$ and $f \in \Fcal_\alpha$;
		\item $\Gamma \subseteq \Pcal$ is countable and $f$ fulfills every member of $\Gamma$;
		\item $\vec x$ is a tuple in 
		$T_\delta^{<\omega}$ with unique drop-downs to $\alpha$;
		\item $f$ is consistent with $\vec x \res \alpha$.
	\end{enumerate}
	Then there exists $g \in \Fcal_\delta$ such that:
	\begin{enumerate}
		\item[(a)] $f \subseteq g$;
		\item[(b)] $g$ fulfills every member of $\Gamma$;
		\item[(c)] $g$ is consistent with $\vec x$.
	\end{enumerate}
\end{prop}

\begin{proof}
	The proof is by induction on $\delta$. 
	The base case $\delta = \alpha$ is immediate.

	\emph{Successor case:} Assume that the statement is true for some $\xi \ge \alpha$, 
	and we prove that it is true for $\delta = \xi + 1$. 
	By the inductive hypothesis, fix $g \in \Fcal_{\xi}$ 
	such that $f \subseteq g$, 
	$g$ fulfills every member of $\Gamma$, and 
	$g$ is consistent with $\vec x \res \xi$. 
	Now apply Lemma \ref{successor case 2} (letting $\vec x_0 = \vec x_1$ 
	be the empty tuple) to fix $h \in \Fcal_{\xi+1}$ such that $g \subseteq h$, 
	$h$ fulfills every member of $\Gamma$, and 
	$h$ is consistent with $\vec x$.

	\emph{Limit case:} Let $\delta > \alpha$ be a countable limit ordinal and 
	assume that the statement holds for all $\xi$ with $\alpha \le \xi < \delta$. 
	Fix an enumeration $\langle P_n : n < \omega \rangle$ of $\Gamma$ 
	in which each member appears infinitely many times, 
	fix an enumeration $\langle c_n : n < \omega \rangle$ of $T \res \delta$ 
	in which each member appears infinitely many times, 
	and fix a strictly increasing sequence of ordinals $\langle \delta_n : n < \omega \rangle$ 
	cofinal in $\delta$ with $\delta_0 = \alpha$. 
	We define by induction $\vec x_n$ and $f_n$ satisfying the 
	following inductive hypotheses:
	\begin{itemize}
		\item $\vec x_n$ is an injective tuple in $T_\delta^{<\omega}$ 
		with meets in $X$;
		\item $\vec x_n$ has unique drop-downs to $\delta_n$;
		\item $f_n \in \Fcal_{\delta_n}$;
		\item $f_n$ is consistent with $\vec x_n \res \delta_n$;
		\item $f_n$ fulfills every member of $\Gamma$.
	\end{itemize}

	For the base case, let $\vec x_0 = \vec x$ and $f_0 = f$. 
	Now assume that $n < \omega$ and for all $m \le n$, 
	$\vec x_m$ and $f_m$ are defined as required. 
	Write $\vec x_n = (a_0,\ldots,a_{l-1})$. 
	Applying the inductive hypothesis, fix $f_{n+1} \in \Fcal_{\delta_{n+1}}$ 
	such that $f_{n} \subseteq f_{n+1}$, 
	$f_{n+1}$ fulfills every member of $\Gamma$, and 
	$f_{n+1}$ is consistent with $\vec x_n \res \delta_{n+1}$.

	We define $\vec x_{n+1}$ in two steps. 
	First, we ask whether $c_n$ has height less than or equal to $\delta_{n+1}$. 
	If not, then we are not yet ready to handle $c_n$, so we move on to step two. 
	Suppose that it does. 
	If $f_{n+1}(c_n) = 0$, then $c_n$ has been handled, 
	and we move on to step two. 
	Suppose that $f_{n+1}(c_n) = 1$. 
	Fix $d \in T_{\delta_{n+1}}$ such that $c \le_T d$ and $f_{n+1}(d) = 1$. 
	If $d <_T a_k$ for some $k < l$, then $c_n$ is already handled and we move on to step two. 
	Otherwise, fix some $e \in T_{\delta}$ above $d$, and let 
	$\vec x_n'$ enumerate the elements of $\vec x_n$ together with $e$. 
	In any of the previous cases mentioned other than the last case, 
	let $\vec x_n' = \vec x_n$.

	Now we describe step two. 
	Since $f_{n+1}$ fulfills $P_n$, we can fix some 
	$\vec w \in P_n$ of height $\delta_{n+1}$ 
	which is consistent with $f_{n+1}$ and disjoint from $\vec x_n' \res \delta_{n+1}$. 
	Fix $\vec w^+ \in P_n$ above $\vec w$ of height $\delta$. 
	Since $f_{n+1}$ is consistent with both $\vec w$ and $\vec x_n' \res \delta_{n+1}$, 
	$\vec w$ and $\vec x_n' \res \delta_{n+1}$ are compatible. 
	By Lemma \ref{compatible drop 2}, $\vec w^+$ and $\vec x_n'$ are compatible. 
	Let $\vec x_{n+1}$ be an injective tuple which lists the elements of 
	$\vec w^+$ and $\vec x_n'$. 

	This completes the induction. 
	Define $g \in \Fcal_{\delta}$ by letting $g \res (T \res \delta) = \bigcup_n f_n$ 
	and $g(y) = 1$ iff for some $n < \omega$, $y$ appears in $\vec x_n$.
\end{proof}

\begin{lemma}[Splitting Extension, 2] \label{splitting extension 2}
	Assume the following:
	\begin{enumerate}
		\item $\alpha \le \delta < \omega_1$ and $f \in \Fcal_\alpha$;
		\item $\Gamma \subseteq \Pcal$ is countable and $f$ fulfills every member of $\Gamma$;
		\item $\vec x$, $\vec y$, and $\vec z$ are injective tuples 
		in $T_\delta^{<\omega}$;
		\item $\vec x$, $\vec y$, and $\vec z$ are pairwise disjoint and pairwise compatible;
		\item $\vec y \res \alpha = \vec z \res \alpha$;
		\item $\vec x \res \alpha$ is disjoint from $\vec y \res \alpha$;
		\item $f$ is consistent with $\vec x \res \alpha$ and $\vec y \res \alpha$;
		\item $\vec y$ and $\vec z$ have uniform meets.
	\end{enumerate}
	Then there exists $h \in \Fcal_\delta$ such that:
	\begin{enumerate}
		\item[(a)] $f \subseteq h$;
		\item[(b)] $h$ fulfills every member of $\Gamma$;
		\item[(c)] $h$ is consistent with $\vec x$, $\vec y$, and $\vec z$.
	\end{enumerate}
\end{lemma}

\begin{proof}
	Let $\gamma = \Delta(\vec y,\vec z)$. 
	Then $\alpha \le \gamma < \delta$ and $\vec y \res \gamma = \vec z \res \gamma$. 
	Since $\vec x$ and $\vec y$ are compatible, by Lemma \ref{compatible drop 2} so are 
	$\vec x \res \gamma$ and $\vec y \res \gamma$. 
	Let $\vec d$ be an injective tuple which lists the elements of 
	$\vec x \res \gamma$ and $\vec y \res \gamma$. 
	As $f$ is consistent with $\vec x \res \alpha$ and $\vec y \res \alpha$, it 
	is also consistent with $\vec d \res \alpha$. 
	By Proposition \ref{extension 2}, we can fix $g \in \Fcal_\gamma$ 
	such that $f \subseteq g$, $g$ fulfills every member of $\Gamma$, 
	and $g$ is consistent with $\vec d$. 
	It follows that $g$ is consistent with both 
	$\vec x \res \gamma$ and $\vec y \res \gamma$. 
	Because $\vec y \res \gamma = \vec z \res \gamma$, 
	$g$ is also consistent with $\vec z \res \gamma$. 

	Applying Lemma \ref{successor case 2} and the fact that 
	$\gamma = \Delta(\vec y,\vec z)$, we can find 
	$g^+ \in \Fcal_{\gamma+1}$ such that 
	$g \subseteq g^+$, $g^+$ fulfills every member of $\Gamma$, 
	and $g^+$ is consistent with $\vec x \res (\gamma+1)$, 
	$\vec y \res (\gamma+1)$, and $\vec z \res (\gamma+1)$. 
	If $\delta = \gamma+1$, then we are done. 
	Assume that $\gamma+1 < \delta$. 
	Let $\vec e$ be an injective tuple which lists the elements of 
	$\vec x$, $\vec y$, and $\vec z$. 
	Since $\vec x$, $\vec y$, and $\vec z$ are pairwise compatible, 
	$\vec e$ has meets in $X$. 
	Also, $g^+$ is consistent with $\vec e \res (\gamma+1)$. 
	Applying Proposition \ref{extension 2}, fix $h \in \Fcal_{\delta}$ such that 
	$g^+ \subseteq h$, $h$ fulfills every member of $\Gamma$, 
	and $h$ is consistent with $\vec e$. 
	Then $h$ is consistent with $\vec x$, $\vec y$, and $\vec z$.
\end{proof}

\section{Adding Promises, 2} \label{Adding Promises, 2}

In this section, we prove a result which implies 
that, under some circumstances, a promise can be added to a 
condition.

\begin{defn}
	For any $f \in \Fcal$ and $P \in \Pcal$, we say that 
	\emph{$P$ is suitable for $f$} if there exist $\alpha < \delta < \omega_1$, 
	$\langle \delta_n : n < \omega \rangle$, 
	and $\langle \vec x_s : s \in {}^{\underaccent{\breve}{\omega}} \omega \rangle$ satisfying:
	\begin{enumerate}
		\item $f \in \Fcal_\alpha$ and $P \in \Pcal_\delta$;
		\item $\langle \delta_n : n < \omega \rangle$ is strictly increasing and cofinal 
		in $\delta$ with $\delta_0 = \alpha$;
		\item for all $n < \omega$, 
		$\{ \vec x_s : s \in {}^n \omega \}$ is a 
		family of injective tuples in 
		$T_{\delta_n}^{\otimes m}$, where $m$ 
		is the dimension of $P$, which is pairwise disjoint, pairwise compatible, 
		and any two elements of the family have uniform meets;
		\item for all $s \subseteq t$ in ${}^{\underaccent{\breve}{\omega}} \omega$, 
		$\vec x_t \res \delta_{|s|} = \vec x_s$;
		\item for all $s \in {}^{\underaccent{\breve}{\omega}} \omega$, there exists 
		$\vec x_s^+ \in P \cap T_\delta^{\otimes m}$ 
		such that $\vec x_s^+ \res \delta_{|s|} = \vec x_s$;
		\item $f$ is consistent with $\vec x_\emptyset$.
	\end{enumerate}
\end{defn}

\begin{lemma}[Existence of Promises, 2] \label{existence of promises 2}
	Assume the following:
	\begin{enumerate}
		\item $\alpha < \omega_1$ and $f \in \Fcal_\alpha$,
		\item $P \in \Pcal$;
		\item there exists $\vec x \in P$ such that $\vec x$ has unique 
		drop-downs to $\alpha$ and $f$ is consistent with $\vec x \res \alpha$.
	\end{enumerate}
	Then there exists $Q \subseteq P$ in $\Pcal$ such that $Q$ is suitable for $f$.
\end{lemma}

\begin{proof}
	Using Proposition \ref{promise existence start 2}, construct by induction 
	sequences $\langle \delta_n : n < \omega \rangle$ and 
	$\langle \vec x_s : s \in {}^{\underaccent{\breve}{\omega}} \omega \rangle$ satisfying:
	\begin{enumerate}
		\item $\langle \delta_n : n < \omega \rangle$ is a strictly increasing sequence 
		of countable ordinals with $\delta_0 = \alpha$;
		\item for all $n < \omega$, 
		$\{ \vec x_s : s \in {}^n \omega \}$ is a family of injective tuples in 
		$T_{\delta_n}^{\otimes m}$, where $m$ 
		is the dimension of $P$, which is pairwise disjoint, 
		pairwise compatible, and any two elements of the family have uniform meets;
		\item for all $s \subseteq t$ in ${}^{\underaccent{\breve}{\omega}} \omega$, 
		$\vec x_t \res \delta_{|s|} = \vec x_s$;
		\item $\vec x_{\emptyset} = \vec x \res \alpha$.
	\end{enumerate}
	For the base case, let $\delta_0 = \alpha$ 
	and $\vec x_{\emptyset} = \vec x \res \alpha$. 
	Now let 
	$n < \omega$ and assume that $\delta_n$ and $\{ \vec x_s : s \in {}^n \omega \}$ 
	are defined as required. 
	For each $s \in {}^n \omega$, 
	apply Proposition \ref{promise existence start 2} 
	to fix a pairwise disjoint and pairwise compatible family 
	$\{ \vec x_{s,l} : l < \omega \}$ of members of $P$ above $\vec x_s$ 
	with the same height 
	such that any two elements of the family have uniform meets. 
	Now fix $\delta_{n+1} < \omega_1$ which is greater than the height of 
	$\vec x_{s,l}$ for all $s \in {}^n \omega$ and $l < \omega$. 
	Then for each such $s$ and $l$, choose $\vec x_{s^{\frown} l}$ in $P$ 
	above $\vec x_{s,l}$ with height $\delta_{n+1}$. 
	Let $\delta = \sup_n \delta_n$. 
	For each $s \in {}^{\underaccent{\breve}{\omega}} \omega$, pick some $\vec x_s^+ \in P$ with height $\delta$ 
	above $\vec x_s$. 
	Now define $P = \bigcup \{ P_{\vec x_s^+} : s \in {}^{\underaccent{\breve}{\omega}} \omega \}$.
\end{proof}

\begin{prop}[Adding Promises, 2] \label{adding promises 2}
	Assume the following:
	\begin{enumerate}
		\item $\alpha < \delta < \omega_1$;
		\item $f \in \Fcal_\alpha$ and $P \in \Pcal_\delta$;
		\item $\Gamma \subseteq \Pcal$ is countable and $f$ fulfills every member of $\Gamma$;
		\item $\vec x$ is an injective tuple in $T_\delta^{<\omega}$ with 
		unique drop-downs to $\alpha$;
		\item $P$ is suitable for $f$ as witnessed by $\langle \delta_n : n < \omega \rangle$ 
		and $\langle \vec x_s : s \in {}^{\underaccent{\breve}{\omega}} \omega \rangle$;
		\item $\vec x \res \alpha = \vec x_\emptyset$.
	\end{enumerate}
	Then there exists $g \in \Fcal_\delta$ such that:
	\begin{enumerate}
		\item[(a)] $f \subseteq g$;
		\item[(b)] $g$ is consistent with $\vec x$;
		\item[(c)] $g$ fulfills every member of $\Gamma \cup \{ P \}$.
	\end{enumerate}
\end{prop}

\begin{proof}
	Let $m$ be the dimension of $P$. 
	Fix an enumeration $\langle P_n : n < \omega \rangle$ 
	of $\Gamma$ in which each member appears infinitely many times and 
	fix an enumeration $\langle c_n : n < \omega \rangle$ of $T \res \delta$ in which 
	every element appears infinitely many times. 
	We define by induction $\vec z_n$, $s_n$, and $f_n$ for all $n < \omega$, 
	maintaining the following inductive hypotheses:
	\begin{itemize}
		\item $\vec z_n$ is an injective tuple in $T_\delta^{<\omega}$ 
		with meets in $X$ and unique drop-downs to $\delta_n$;
		\item $s_n \in {}^{n} \omega$, $\vec x_{s_n}$ has height $\delta_n$, 
		and for $n > 0$, $\vec x_{s_n}$ and $\vec z_n \res \delta_n$ are disjoint;
		\item $k < n$ implies $s_k \subseteq s_n$;
		\item $f_n \in \Fcal_{\delta_n}$;
		\item $f_n$ is consistent with $\vec z_n \res \delta_n$ and $\vec x_{s_n}$;
		\item $f_n$ fulfills every member of $\Gamma$.
	\end{itemize}

	For the base case, let $\vec z_0 = \vec x$, $s_0 = \emptyset$, 
	and $f_0 = f$. 
	Now assume that $n < \omega$ and for all $m \le n$,  
	$\vec z_m$, $s_m$, and $f_m$ are defined as required. 
	Write 
	$\vec z_n = (a_0,\ldots,a_{l-1})$. 
	Pick $s_{n,0}$ and $s_{n,1}$ in ${}^{n+1} \omega$ such that for each $i < 2$, 
	$s_n \subseteq s_{n,i}$ and $\vec x_{s_{n,i}}$ is disjoint from 
	$\vec z_n \res \delta_{n+1}$. 
	Define $s_{n+1} = s_{n,1}$.
	
	Apply Lemma \ref{splitting extension 2} to find $f_{n+1} \in \Fcal_{\delta_{n+1}}$ 
	such that $f_n \subseteq f_{n+1}$, $f_{n+1}$ fulfills every member of $\Gamma$, 
	and $f_{n+1}$ is consistent with $\vec z_n \res \delta_{n+1}$, 
	$\vec x_{s_{n,0}}$, and $\vec x_{s_{n,1}}$. 
	Fix $\vec x_{s_{n,0}}^+$ in $P \cap T_\delta^{\otimes m}$ 
	such that $\vec x_{s_{n,0}}^+ \res \delta_{n+1} = \vec x_{s_{n,0}}$. 
	Since $f_{n+1}$ is consistent with $\vec z_n \res \delta_{n+1}$ and $\vec x_{s_{n,0}}$, 
	by Lemma \ref{compatible drop 2} we have that $\vec z_n$ and $\vec x_{s_{n,0}}^+$ 
	are compatible. 
	Let $\vec z_n'$ be an injective tuple which lists the elements 
	of $\vec z_n$ and $\vec x_{s_{n,0}}^+$.

	We define $\vec z_{n+1}$ in two steps. 
	First, we ask whether $c_n$ has height less than or equal to $\delta_{n+1}$. 
	If not, then we are not yet ready to handle $c_n$, so we move on to step two. 
	Assume that it does. 
	If $f_{n+1}(c_n) = 0$, then $c_n$ has been handled, and we move on to step two. 
	Suppose that $f_{n+1}(c_n) = 1$. 
	Fix $d \in T_{\delta_{n+1}}$ with $d \ge_T c_n$ such that $f_{n+1}(d) = 1$. 
	If $d$ is below some member of $\vec z_n'$, 
	then $c_n$ is already handled and we move on to step two. 
	Otherwise, fix some $e \in T_{\delta}$ above $d$, and let 
	$\vec z_n''$ enumerate the elements of $\vec z_n'$ together with $e$. 
	In any of the cases mentioned other than the last case, let $\vec z_n'' = \vec z_n'$.
	
	Now we describe step two. 
	As $f_{n+1}$ fulfills $P_n$, we can 
	fix some $\vec w \in P_n$ of height $\delta_{n+1}$ 
	which is disjoint from $\vec z_n'' \res \delta_{n+1}$ such that 
	$f_{n+1}$ is consistent with $\vec w$. 
	Fix $\vec w^+ \in P_n$ above $\vec w$ with height $\delta$. 
	Since $f_{n+1}$ is consistent with both $\vec w$ and $\vec z_n'' \res \delta_{n+1}$, 
	$\vec w$ and $\vec z_n'' \res \delta_{n+1}$ are compatible. 
	By Lemma \ref{compatible drop 2}, $\vec w^+$ and $\vec z_n''$ are compatible. 
	Let $\vec z_{n+1}$ be an injective tuple which lists the elements of 
	$\vec w^+$ and $\vec z_n''$. 

	This completes the induction. 
	Define $g \in \Fcal_{\delta}$ by letting $g \res (T \res \delta) = \bigcup_n f_n$ 
	and $g(y) = 1$ iff for some $n < \omega$, $y$ appears in $\vec z_n$.
\end{proof}

\section{The Second Poset} \label{The Second Poset}

We are now ready to introduce and derive the main properties of 
the forcing which we use to prove Theorem 2.

\begin{defn}
	Define $\p(T,X)$ to be the forcing consisting of conditions which are 
	ordered pairs $p = (f_p,\Gamma_p)$ satisfying:
	\begin{enumerate}
		\item $f_p \in \Fcal_{\alpha_p}$ for some $\alpha_p < \omega_1$;
		\item $\Gamma_p \subseteq \Pcal$ is countable;
		\item $f_p$ fulfills every member of $\Gamma_p$.
	\end{enumerate}
	Define $q \le p$ if $f_p \subseteq f_q$ and $\Gamma_p \subseteq \Gamma_q$.
\end{defn}

The next two lemmas follow immediately from 
Propositions \ref{extension 2} and \ref{adding promises 2} respectively.

\begin{lemma}[Poset Extension, 2] \label{extension poset 2}
	Assume the following:
	\begin{enumerate}
		\item $p \in \p(T,X)$;
		\item $\alpha_p < \delta < \omega_1$;
		\item $\vec x$ is an injective tuple with height $\delta$ 
		and unique drop-downs to $\alpha_p$;
		\item $f_p$ is consistent with $\vec x \res \alpha_p$.
	\end{enumerate}
	Then there exists $q \le p$ such that $\alpha_q = \delta$ and $f_q$ 
	is consistent with $\vec x$.
\end{lemma}

It follows that, assuming $\p(T,X)$ preserves $\omega_1$, 
any generic filter on $\p(T,X)$ yields a subtree which witnesses that 
$X \in \Ucal(T)$, as described in 
Section \ref{Making U(T) an Ultrafilter}.

\begin{lemma}[Poset Adding Promises, 2] \label{adding promises to conditions 2}
	Assume the following:
	\begin{enumerate}
		\item $p \in \p(T,X)$;
		\item $\alpha_p < \delta < \omega_1$;
		\item $P \in \Pcal_\delta$;
		\item $\vec x$ is an injective tuple of height $\delta$ with 
		unique drop-downs to $\alpha_p$;
		\item $f_p$ is consistent with $\vec x \res \alpha_p$;
		\item $P$ is suitable for $f$ with witness $\{ \vec x_s : s \in {}^{\underaccent{\breve}{\omega}} \omega \}$;
		\item $\vec x \res \alpha_p = \vec x_{\emptyset}$.
	\end{enumerate}
	Then there exists $q \le p$ such that $\alpha_q = \delta$, 
	$f_q$ is consistent with $\vec x$, and $P \in \Gamma_q$.
\end{lemma}

\begin{lemma} \label{extending into dense sets 2}
	Let $\lambda \ge (2^{|\p(T,X)|})^+$ be a regular cardinal and let 
	$N \prec H(\lambda)$ be countable such that $T$, $X$, and $\p(T,X)$ are in $N$. 
	Let $\delta = N \cap \omega_1$. 
	Assume the following:
	\begin{enumerate}
		\item $p \in N \cap \p(T,X)$;
		\item $D \in N$ is a dense subset of $\p(T,X)$;
		\item $\vec x$ is an injective tuple of height $\delta$ 
		with unique drop-downs to $\alpha_p$;
		\item $f_p$ is consistent with $\vec x \res \alpha_p$.
	\end{enumerate}
	Then there exists $q \le p$ such that $q \in N \cap D$ and 
	$f_q$ is consistent with $\vec x \res \alpha_q$.
\end{lemma}

\begin{proof}
	Suppose for a contradiction that the conclusion of the lemma fails. 
	Let $n$ be the dimension of $\vec x$. 
	Define $P_0$ as the set of all 
	$\vec y \in (T^{\otimes n})_{\vec x \res \alpha_p}$ 
	such that for any $q \le p$ in $D$ with $\alpha_q$ less than 
	or equal to the height of $\vec y$, 
	$f_q$ is not consistent with $\vec y \res \alpha_q$. 
	Note that $P_0 \in N$ by elementarity, 
	and $P_0$ is a downwards closed subset of the tree 
	$(T^{\otimes n})_{\vec x \res \alpha_p}$. 
	We claim that for all $\alpha_p \le \gamma < \delta$, 
	$\vec x \res \gamma \in P_0$. 
	Otherwise, there exists $q \le p$ in $D$ such that $\alpha_q \le \gamma$ 
	and $f_q$ is consistent with $(\vec x \res \gamma) \res \alpha_q = 
	\vec x \res \alpha_q$. 
	Since $p$, $D$, $\gamma$, and $\vec x \res \gamma$ are in $N$, by elementarity 
	we may assume that $q \in N$. 
	But then the conclusion of the lemma holds, which is a contradiction. 
	It follows from the claim that $P_0$ is uncountable. 
	Let $P \in N$ be a non-empty downwards closed subset of $P_0$ such that 
	every element of $P$ has uncountably many members of $P$ above it. 
	Note that $\vec x \res \alpha_p \in P$ and $P \in \Pcal$.

	By Lemma \ref{existence of promises 2}, there exists $Q \in \Pcal$ 
	such that $Q \subseteq P$ and $Q$ is suitable for $f_p$. 
	Let $\gamma$ be the base level of $Q$. 
	Since the root of $P_0$ is $\vec x \res \alpha_p$, 
	$\vec x \res \alpha_p$ is the only element of $P$ of height $\alpha_p$. 
	Hence, $\vec x \res \alpha_p$ is equal to $\vec x_\emptyset$, where 
	$\{ \vec x_s : s \in {}^{\underaccent{\breve}{\omega}} \omega \}$ 
	witnesses that $Q$ is suitable for $f_p$. 
	By Lemma \ref{adding promises to conditions 2}, we can fix $q \le p$ 
	such that $Q \in \Gamma_q$. 
	As $D$ is dense, we can fix $r \le q$ in $D$. 
	Because $f_r$ fulfills $Q$, there exists some 
	$\vec y \in Q \cap (T_{\alpha_r})^{n}$ 
	such that $f_r$ is consistent with $\vec y$. 
	Then $\vec y \in P_0$. 
	But $r \le p$ is in $D$ and $f_r$ is consistent 
	with $\vec y$, contradicting the definition of $P_0$.
\end{proof}

\begin{thm} \label{totally proper 2}
	The forcing poset $\p(T,X)$ is totally proper.
\end{thm}

\begin{proof}
	Let $\lambda \ge (2^{|\p(T,X)|})^+$ be a regular cardinal and let 
	$N \prec H(\lambda)$ be countable such that $T$, $X$, and $\p(T,X)$ are in $N$. 
	Let $\delta = N \cap \omega_1$. 
	Consider $p \in N \cap \p(T,X)$. 
	We claim that there exists $q \le p$ such that $q$ is a total master condition 
	for $N$ and $\p(T,X)$. 
	Specifically, we build by induction a descending sequence of conditions 
	$\langle p_n : n < \omega \rangle$ in $N \cap \p(T,X)$ such that for every 
	dense set $D \in N$, there exists some $n$ with $p_n \in D$, and then 
	define a lower bound $q$ of the sequence.
	
	Fix the following objects:
	\begin{enumerate}
		\item an enumeration $\langle P_n : n < \omega \rangle$ of all members of $\Pcal$ 
		which lie in $N$ such that every element appears infinitely many times;
		\item an enumeration $\langle c_n : n < \omega \rangle$ of $T \res \delta$ 
		in which each member appears infinitely many times;
		\item a sequence $\langle D_n : n < \omega \rangle$ which lists all dense subsets 
		of $\p(T,X)$ which lie in $N$.
	\end{enumerate}
	We define by induction a descending sequence $\langle p_n : n < \omega \rangle$ 
	of conditions in $N \cap \p(T,X)$ together with a sequence 
	$\langle \vec x_n : n < \omega \rangle$ of injective tuples in $T_\delta^{<\omega}$ 
	as follows.
	
	For the base case, let $p_0 = p$ and let $\vec x_0$ be the empty tuple. 
	Now assume that $p_n$ and $\vec x_n$ are defined such that $\vec x_n$ has 
	unique drop-downs to $\alpha_{p_n}$. 
	We define $p_{n+1}$ and $\vec x_{n+1}$ in several steps. 
	In the first step, apply Lemma \ref{extending into dense sets 2} to find 
	$p_{n+1} \le p_n$ in $N \cap D_n$ such that 
	$f_{p_{n+1}}$ is consistent with $\vec x_n \res \alpha_{p_{n+1}}$. 
	
	In the second step, 
	we ask whether $c_n$ has height less than or equal to $\alpha_{p_{n+1}}$. 
	If not, then we are not yet ready to handle $c_n$, and we move on to step three. 
	Suppose that it does. 
	If $f_{p_{n+1}}(c_n) = 0$, then $c_n$ has been handled, so we move on to step three. 
	Suppose that $f_{n+1}(c_n) = 1$. 
	Fix $d \in T_{\alpha_{p_{n+1}}}$ such that $c \le_T d$ and $f_{p_{n+1}}(d) = 1$. 
	If $d$ is below some member of $\vec x_n$, 
	then $c_n$ has already handled and we move on to step three. 
	Otherwise, fix some $e \in T_{\delta}$ above $d$, and let 
	$\vec x_n'$ injectively enumerate the elements of $\vec x_n$ together with $e$. 
	In any of the previous cases mentioned other than the last case, 
	let $\vec x_n' = \vec x_n$.

	Now we describe step three. 
	Consider $P_n$. 
	If $P_n$ is not in $\Gamma_{p_{n+1}}$, then let $\vec x_{n+1} = \vec x_n'$. 
	Suppose that $P_n$ is in $\Gamma_{p_{n+1}}$. 
	Then $f_{p_{n+1}}$ fulfills $P_n$, so we can fix some 
	$\vec y \in P_n$ of height $\alpha_{p_{n+1}}$ 
	which is consistent with $f_{p_{n+1}}$ and disjoint from $\vec x_n' \res \alpha_{p_{n+1}}$. 
	Fix $\vec y^+ \in P_n$ above $\vec y$ of height $\delta$. 
	Since $f_{p_{n+1}}$ is consistent with both $\vec y$ and $\vec x_n' \res \alpha_{p_{n+1}}$, 
	$\vec y$ and $\vec x_n' \res \alpha_{p_{n+1}}$ are compatible. 
	By Lemma \ref{compatible drop 2}, $\vec y^+$ and $\vec x_n'$ are compatible. 
	Let $\vec x_{n+1}$ be an injective tuple which lists the elements of 
	$\vec y^+$ and $\vec x_n'$. 

	This completes the induction. 
	Now define $q$ as follows. 
	Let $f_{q} \res (T \res \delta) = \bigcup_n f_{p_n}$, and 
	for all $z \in T_\delta$, $f_q(z) = 1$ iff for some $n < \omega$, 
	$z$ appears in $\vec x_n$. 
	Let $\Gamma_q = \bigcup_n \Gamma_{p_n}$. 
	Then $q = (f_q,\Gamma_q)$ is in $\p(T,X)$ and is a lower bound of 
	$\langle p_n : n < \omega \rangle$.
\end{proof}

\section{Further Results and Open Problems}

Abraham and Shelah \cite{AS2} introduced a variation of the promise method 
which allows for the preservation of Suslin trees. 
Using this alternative version of promises, 
we can prove:

\begin{thm} \label{suslin versions}
	Each of the following statements is consistent with \textsf{CH} 
	and the existence of a Suslin tree:
	\begin{enumerate}
		\item any two Countryman lines are compatible;
		\item for any coherent Aronszajn tree 
		$T \subseteq {}^{\underaccent{\breve}{\omega}_1} \omega$ 
		which is non-Suslin, $\mathcal U(T)$ is an ultrafilter.
	\end{enumerate}
\end{thm} 

Statement (1) does not need any adjustment in the presence of a Suslin tree, since 
lexicographically ordered Aronszajn trees which are Countryman lines are already 
non-Suslin.\footnote{Since we are unable to find a proof of this fact in the literature, 
we give a brief sketch. By passing to a subtree, we may assume that our 
Aronszajn tree $T$ with lexicographical ordering $L$ 
satisfies that the set of successors of any element at any higher level 
is isomorphic to the rationals (use Lemma \ref{club lemma} and the argument 
of \cite[Ch.\ 4, Prop.\ 2]{baumgartnerdiss}). 
Given an uncountable set $A \subseteq T$, find a sequence 
$\langle (a_i,b_i) : i < \omega_1 \rangle$ such that for all $i$, 
$a_i \in A$, $a_i$ and $b_i$ 
have the same height, $a_i <_L b_i$, and $a_i$ has height less than 
$\Delta(a_j,b_j)$ for all $j > i$. 
Then for all $i < j$, if $a_i <_T a_j$, then $a_i <_T b_j$, so 
$b_j <_L b_i$ by Lemma \ref{easy lex lemma}. 
Since $L$ is Countryman, fix an uncountable set 
$X \subseteq \omega_1$ such that 
$\{ (a_i,b_i) : i \in X \}$ is a chain in $L^2$. 
Then for all $i < j$, $a_i$ and $a_j$ are incomparable in $T$, 
for otherwise $a_i <_L a_j$ but $b_j <_L b_i$.} 
The proof of Theorem \ref{suslin versions} is similar to \cite[Th.\ 8.1]{jk45} 
and will appear in the second author's PhD disseration. 

Recall Problem 1 which is stated in the introduction: 
Is it consistent with \textsf{CH} that for any Aronszajn lines $L$ and $J$, 
there exists a Countryman line $C$ such that each of $L$ and $J$ contain 
either a copy of $C$ or its converse $C^*$? 
This question can be approached in a series of steps, which we 
outline below. 
The most difficult to prove affirmatively is almost certainly Problem 2.

\begin{probnumber} \label{problem2}
	Is it consistent with \textsf{CH} that there exists an Aronszajn tree $T$ 
	which satisfies \textsf{CAT}?
\end{probnumber}

In light of \cite{linear_basis}, a natural candidate for such a tree would be a 
normal coherent Aronszajn subtree of ${}^{\underaccent{\breve}{\omega}_1} 2$. 
Recall that assuming mild forcing axioms, the existence of an Aronszajn tree satisfying 
\textsf{CAT} is equivalent to \textsf{CAT} holding for all Aronszajn trees. 
Using arguments as in \cite{jk45}, 
it seems likely that the same could be arranged in the 
\textsf{CH} context.

\begin{probnumber}
	Is it consistent with \textsf{CH} that \textsf{CAT} holds?
\end{probnumber}

\begin{probnumber}
	Is it consistent with \textsf{CH} that every Aronszajn line contains 
	a Countryman line?
\end{probnumber}

Each of these problems has a variation involving the existence of a Suslin tree, for example, 
by restricting to non-Suslin Aronszajn trees. 
In particular:

\begin{probnumber}
	Is it consistent with \textsf{CH} and the existence of a Suslin tree that 
	any two Aronszajn lines, neither of which contains a Suslin line, are compatible?
\end{probnumber}


\end{document}